\documentclass[1 [leqno,11pt]{amsart}
\usepackage{amssymb, amsmath}
\def\refer#1{~\ref{#1}}
\def\refeq#1{~(\ref{#1})}
\def\ccite#1{~\cite{#1}}

\def\longformule#1#2{
\displaylines{ \qquad{#1} \hfill\cr \hfill {#2} \qquad\cr } }
\def\inte#1{
\displaystyle\mathop{#1\kern0pt}^\circ }

\let\pa=\partial
\let\al=\alpha

\let\d=\delta
\let\e=\varepsilon
\let\ep=\varepsilon

\let\lam=\lambda
\let\r=\rho
\let\s=\sigma
\let\f=\frac

\let\p=\psi

\let\D=\Delta

\let\Om=\Omega
\let\wt=\widetilde
\let\wh=\widehat

\def\cB{{\mathcal B}}
\def\cC{{\mathcal C}}

\def\cF{{\mathcal F}}

\def\cL{{\mathcal L}}
\def\cM{{\mathcal M}}

\def\cR{{\mathcal R}}
\def\cS{{\mathcal S}}

\def\grad{\nabla}
\def\la{\lambda}
\def\fB{\frak{B}}

\def\virgp{\raise 2pt\hbox{,}}
\def\cdotpv{\raise 2pt\hbox{;}}

\def\eqdefa{\buildrel\hbox{\footnotesize def}\over =}
\def\Id{\mathop{\rm Id}\nolimits}

\def\C{\mathop{\mathbb C\kern 0pt}\nolimits}
\def\DD{\mathop{\mathbb D\kern 0pt}\nolimits}
\def\EE{\mathop{{\mathbb E \kern 0pt}}\nolimits}
\def\K{\mathop{\mathbb K\kern 0pt}\nolimits}
\def\N{\mathop{\mathbb N\kern 0pt}\nolimits}
\def\Q{\mathop{\mathbb Q\kern 0pt}\nolimits}
\def\R{\mathop{\mathbb R\kern 0pt}\nolimits}
\def\SS{\mathop{\mathbb S\kern 0pt}\nolimits}
\def\ZZ{\mathop{\mathbb Z\kern 0pt}\nolimits}
\def\TT{\mathop{\mathbb T\kern 0pt}\nolimits}
\def\P{\mathop{\mathbb P\kern 0pt}\nolimits}

\newcommand{\Z}{{\ZZ}}

\def\dv{\mbox{div}}

\def\dive{\mathop{\rm div}\nolimits}

\def\Supp{\mathop{\rm Supp}\nolimits\ }

\def\no{\noindent}
\def\na{\nabla}
\def\p{\partial}

\newcommand{\beq}{\begin{equation}}
\newcommand{\eeq}{\end{equation}}
\newcommand{\ben}{\begin{eqnarray}}
\newcommand{\een}{\end{eqnarray}}
\newcommand{\beno}{\begin{eqnarray*}}
\newcommand{\eeno}{\end{eqnarray*}}
\newcommand{\andf}{\quad\hbox{and}\quad}
\newcommand{\with}{\quad\hbox{with}\quad}
\newtheorem{defi}{Definition}[section]
\newtheorem{thm}{Theorem}[section]
\newtheorem{lem}{Lemma}[section]
\newtheorem{rmk}{Remark}[section]

\renewcommand{\theequation}{\thesection.\arabic{equation}}
\begin{document}
\title[Large solutions of $3-$D inhomogeneous NS equations]
{Global large solutions  to 3-D inhomogeneous
 Navier-Stokes system with one slow variable}
 \author[J.-Y. CHEMIN]{Jean-Yves Chemin}
\address [J.-Y. Chemin]%
{Laboratoire J.-L. Lions, UMR 7598 \\
Universit\'e Pierre et Marie Curie, 75230 Paris Cedex 05, FRANCE }
\email{chemin@ann.jussieu.fr}
\author[M. PAICU]{Marius Paicu}
\address [M. PAICU]%
{Universit\'e  Bordeaux 1\\
 Institut de Math\'ematiques de Bordeaux\\
F-33405 Talence Cedex, France} \email{marius.paicu@math.u-bordeaux1.fr}
\author[P. ZHANG]{Ping Zhang}%
\address[P. ZHANG]
 {Academy of
Mathematics $\&$ Systems Science and  Hua Loo-Keng Key Laboratory of
Mathematics, The Chinese Academy of Sciences\\
Beijing 100190, CHINA } \email{zp@amss.ac.cn}
\date{10/11/2012}
\maketitle
\begin{abstract} In this paper, we are concerned with the global
wellposedness of 3-D  inhomogeneous incompressible Navier-Stokes
equations \eqref{1.3} in  the critical Besov spaces with the norm of
which are invariant by the  scaling of the equations and under a
nonlinear smallness condition on the isentropic critical Besov norm
to the fluctuation of the initial density and the critical
anisotropic Besov norm of the horizontal components of the initial
velocity which have to be exponentially small compared with the
critical anisotropic Besov norm to the third component of the
initial velocity. The novelty of this results is that the isentropic
space structure to the homogeneity of the initial density function
is consistent with the propagation of anisotropic regularity for the
velocity field. In the second part,
 we apply the same idea to prove the
global wellposedness of \eqref{1.3} with some large data which are
slowly varying in one direction.
\end{abstract}


\vskip 0.2cm


\setcounter{equation}{0}
\section{Introduction}

In this paper, we consider the  global wellposeness  to the
following 3-D incompressible inhomogeneous  Navier-Stokes equations
with initial data in the critical  Besov spaces and with the third
component of the initial velocity being large:
\begin{equation}
 \left\{\begin{array}{l}
\displaystyle \pa_t \rho + \dv(\rho u)=0,\qquad (t,x)\in\R^+\times\R^3, \\
\displaystyle \pa_t (\rho u) + \dv (\rho u\otimes u) -\mu \D u+\grad\Pi=0, \\
\displaystyle \dv\, u = 0, \\
\displaystyle \rho|_{t=0}=\rho_0,\quad \rho u|_{t=0}=\r_0u_0,
\end{array}\right. \label{1.1}
\end{equation}
where $\rho, u=(u_1,u_2, u_3)$ stand for the density and  velocity
of the fluid respectively,  $\Pi$ is a scalar pressure function.
  Such system describes a fluid which is obtained by
mixing two immiscible fluids that are incompressible and that have
different densities. It may also describe a fluid containing a
melted substance.

 In\ccite{LS}, O. Lady\v zenskaja and V. Solonnikov
first addressed  the question of unique resolvability of
(\ref{1.1}). More precisely, they  considered the system \eqref{1.1}
in bounded domain $\Om$ with homogeneous Dirichlet boundary
condition for $u.$ Under the assumption that $u_0$ belongs to~$
W^{2-\frac2p,p}(\Om)$ with $p$ greater than~$d$,  is divergence free and vanishes on
$\p\Om$ and that $\r_0$ is~$C^1(\Om)$,  bounded and  away from zero, then
they  proved
\begin{itemize}
\item Global well-posedness in dimension $d=2;$
\item Local well-posedness in dimension $d=3.$ If in addition $u_0$ is small in $W^{2-\frac2p,p}(\Om),$
then global well-posedness holds true.
\end{itemize}
 Similar results were obtained by R. Danchin
\cite{danchin2} in $\R^d$ with initial data in the almost critical
Sobolev spaces. In  general, the global existence of weak solutions
with finite energy was established by P.-L. Lions in \cite{ LP} (see also
the reference therein, and the monograph \cite{AKM}).   H. Abidi, G. Gui
and the last author  established  in\ccite{A-G-Z}the large time decay and
stability to any given global smooth solutions of \eqref{1.1}.

When the initial density is away from zero, we denote by
$a\eqdefa\frac{1}{\rho}-1,$ and then \eqref{1.1} can be equivalently
formulated as
\begin{equation}\label{1.3}
 \quad\left\{\begin{array}{l}
\displaystyle \pa_t a + u \cdot \grad a=0,
\hspace{1cm}(t,x)\in \R^+\times\R^3,\\
\displaystyle \pa_t u + u \cdot \grad u+ (1+a)(\grad \Pi-\mu\Delta u)=0, \\
\displaystyle \dv\, u = 0, \\
\displaystyle (a, u)|_{t=0}=(a_0, u_{0}).
\end{array}\right.
\end{equation}
Notice that just as the classical Navier-Stokes system (which
corresponds to  $a=0$ in \eqref{1.3}), the inhomogeneous
Navier-Stokes system (\ref{1.3}) also has a scaling. Indeed if $(a,
u)$ solves (\ref{1.3}) with initial data $(a_0, u_0)$, then for
$\forall \, \ell>0$,
\begin{equation}\label{1.2}
(a, u)_{\ell} \eqdefa (a(\ell^2\cdot, \ell\cdot), \ell u(\ell^2
\cdot, \ell\cdot))\quad\mbox{and}\quad (a_0,u_0)_\ell\eqdefa
(a_0(\ell\cdot),\ell u_0(\ell\cdot))
\end{equation}
 $(a, u)_{\ell}$ is also a solution of (\ref{1.3}) with initial data $(a_0,u_0)_\ell$.

\medbreak

It is easy to check that the norm of
$B_{p,1}^{\frac{d}p}(\R^d)\times B_{p,1}^{-1+\frac{d}p}(\R^d)$ is
scaling invariant under the scaling transformation $(a_0,u_0)_\ell$
given by \eqref{1.2}. In \cite{abidi}, H. Abidi proved in general space
dimension $d$ that: if $1<p<2d,$ $ 0<\underline{\mu}<\mu(\r),$ given
 initial  data  $(a_0,u_0)$ sufficiently small in~$ B_{p,1}^{\frac{d}p}(\R^d)\times B_{p,1}^{-1+\frac{d}p}(\R^d),$ (\ref{1.3}) has a global
solution.
 Moreover, this
solution is unique if $p$ is in~Ê$]1,d[.$  This result generalized the
wellposedness  results of R. Danchin in\ccite{danchin} and\ccite{ danchin2}, which
corresponds to the celebrated results by Fujita and Kato
\cite{fujitakato} devoted to the classical Navier-Stokes system, and
was improved by H. Abidi and the second author in \cite{AP} with  $a_0$ in~$B_{q,1}^{\frac{d}q}(\R^d)$ and $u_0$ in~$ B_{p,1}^{-1+\frac{d}p}(\R^d)$
for $p,q$ satisfying some technical assumptions.  H. Abidi, G. Gui and
the last author removed the smallness condition for $a_0$ in \cite{AGZ2,
AGZ3}. Notice that  the main feature of the density space is to be a
multiplier on the velocity space and this allows to define the
nonlinear terms in the system \eqref{1.1}. Recently, R. Danchin and
P. Mucha  proved in\ccite{DM} a more general wellposedness result of
\eqref{1.1}  by considering very rough densities in some multiplier
spaces on the Besov spaces~$B_{p,1}^{-1+\frac{d}p}(\R^d)$ for~$p$ in~Ê$]1,2d[$ which in particular completes the uniqueness result in
\cite{abidi} for $p$ in~$] d,2d[$ in the constant viscosity case.

Motivated by \cite{GZ2, PZ1, Zhangt} concerning the global
wellposedness of 3-D incompressible aniso\-tropic Navier-Stokes system
with the third component of the initial velocity field being large,
the last two authors relaxed in\ccite{PZ2} the smallness condition in \cite{AP} so that
\eqref{1.3} still has a unique global solution (see Theorem
\ref{thpz1} below for details). We emphasize that the proof in
\cite{PZ2} used in a fundamental way the algebraical structure of
\eqref{1.3}. The first step is to obtain energy estimates on the
horizontal components of the velocity field on the one hand and then
on the vertical component on the other hand. Compared with \cite{
GZ2, PZ1, Zhangt}, the additional difficulties with this strategy
are that: there appears a hyperbolic type equation in \eqref{1.3}
and due to the appearance of $a$ in the momentum equation of
\eqref{1.3}, the pressure term is more difficult to be handled. We
remark that the equation on the vertical component of the velocity
field is a linear equation with coefficients depending on the
horizontal components of the velocity field and $a.$ Therefore, the
equation on the vertical component does not demand any smallness
condition. While the equations on the horizontal components of the
velocity field contain bilinear terms in the horizontal components
and also terms taking into account the interactions between the
horizontal components and the vertical one. In order to solve this
equation, we need a smallness condition on $a$ and the horizontal
component (amplified by the vertical component) of the initial data.
The purpose of this paper is to prove the global wellposedness of
\eqref{1.3} with initial data, $a_0, u_0=(u_0^h,u_0^3),$ satisfying
some nonlinear smallness condition on the critical isentropic Besov
norm to $a_0$ and the critical anisotropic Besov norm to $u_0^h$
which have to be exponentially small in contrast with the  critical
anisotropic Besov norm to $u_0^3.$ Then we apply the same idea to
prove the global wellposedness of \eqref{1.3} with some large data
which are slowly varying in one direction.

Before going further, we recall the functional space framework we
are going to use. As in\ccite{CDGG}, \ccite{CZ1} and \ccite{Pa02},
the definitions of the spaces we are going to work with requires
anisotropic dyadic decomposition   of the Fourier variables. Let us
recall from \cite{BCD} that \beq
\begin{split}
&\Delta_k^ha=\cF^{-1}(\varphi(2^{-k}|\xi_h|)\widehat{a}),\qquad
\Delta_\ell^va =\cF^{-1}(\varphi(2^{-\ell}|\xi_3|)\widehat{a}),\\
&S^h_ka=\cF^{-1}(\chi(2^{-k}|\xi_h|)\widehat{a}), \qquad\ S^v_\ell a
= \cF^{-1}(\chi(2^{-\ell}|\xi_3|)\widehat{a})
 \quad\mbox{and}\\
&\Delta_ja=\cF^{-1}(\varphi(2^{-j}|\xi|)\widehat{a}),
 \qquad\ \
S_ja=\cF^{-1}(\chi(2^{-j}|\xi|)\widehat{a}), \end{split}
\label{1.0}\eeq where $\xi_h=(\xi_1,\xi_2),$ $\cF a$ and
$\widehat{a}$ denote the Fourier transform of the distribution $a,$
$\chi(x)$ and~$\varphi(\tau)$ are smooth functions such that \beno
\begin{split}
&\Supp \varphi \subset \Bigl\{\tau \in \R\,/\  \ \frac34\leq |\tau|
\leq \frac83 \Bigr\}\andf \  \ \forall
 \tau>0\,,\ \sum_{j\in\Z}\varphi(2^{-j}\tau)=1,\\
&\Supp \chi \subset \Bigl\{\tau \in \R\,/\  \ \ |\tau| \leq \frac43
\Bigr\}\quad \ \ \andf \  \ \, \chi(\tau)+ \sum_{j\geq
0}\varphi(2^{-j}\tau)=1.
\end{split}
 \eeno

\begin{defi}\label{def1.1}
{\sl  Let $(p,r)\in[1,+\infty]^2,$ $s\in\R$ and $u\in{\mathcal
S}_h'(\R^3),$ which means that $u\in\cS'(\R^3)$ and
$\lim_{j\to-\infty}\|S_ju\|_{L^\infty}=0,$ we set
$$
\|u\|_{B^s_{p,r}}\eqdefa\Big(2^{qs}\|\Delta_q u\|_{L^{p}}\Big)_{\ell
^{r}}.
$$
\begin{itemize}

\item
For $s<\frac{3}{p}$ (or $s=\frac{3}{p}$ if $r=1$), we define $
B^s_{p,r}(\R^3)\eqdefa \big\{u\in{\mathcal S}_h'(\R^3)\;\big|\; \|
u\|_{B^s_{p,r}}<\infty\big\}.$

\item
If $k\in\N$ and $\frac{3}{p}+k\leq s<\frac{3}{p}+k+1$ (or
$s=\frac{3}{p}+k+1$ if $r=1$), then $ B^s_{p,r}(\R^3)$ is defined as
the subset of distributions $u\in{\mathcal S}_h'(\R^3)$ such that
$\partial^\beta u\in B^{s-k}_{p,r}(\R^3)$ whenever $|\beta|=k.$
\end{itemize}

\noindent
{\bf Notations } In all that follows, we shall denote
$$
\cB^s_p\eqdefa B^{s}_{p,1}.
$$
}
\end{defi}

The following theorem was proved by the last two authors in \cite{PZ2}:
\begin{thm}\label{thpz1}
{\sl Let~$p$ be in~$]1,6[$. There
exist positive constants $c_0$ and $C_0$ such that, for any  data
$a_0$ in~$\cB^{\f3p}_{p}(\R^3)$ and $u_0=(u_0^h,u_0^3)$ in~$
\cB^{-1+\frac3p}_{p}(\R^3)$ verifying  \beq \label{1.4} \eta\eqdefa
\bigl(\mu\|a_0\|_{\cB_{p}^{\f3p}}+\|u_0^h\|_{\cB^{-1+\frac3p}_{p}}\bigr)\exp\Bigl(
\frac {C_0} {\mu^2} \|u_0^3\|_{\cB^{-1+\frac3p}_{p}}^2\ \Big)\leq c_0\mu,
\eeq
the system~\eqref{1.3} has a unique  global solution $(a,u)$ in the space
$$
\cC_b([0,\infty[;\cB^{\f3 p}_{p}(\R^3))\times\bigl(\cC_b([0,\infty) ;\cB^{-1+\frac3p}_{p}(\R^3))
\cap L^1(\R^+,\cB^{1+\frac3p}_{p}(\R^3))\bigr).
$$
}
\end{thm}
We want to prove here an anisotropic version of the above theorem.
Let us define the anisotropic Besov space that we are going to use.

\begin{defi}\label{def1.2}
{\sl  Let $p$ be in~$[1,+\infty] $, $s_1\leq \frac2p$, $s_2\leq\frac1p$ and
$u$ in~${\mathcal S}_h'(\R^3),$ we set
$$
\|u\|_{\fB^{s_1,s_2}_{p}}\eqdefa\Big(2^{js_1}2^{ks_2}\|\Delta^h_j
\Delta_k^vu\|_{L^{p}}\Big)_{\ell ^{1}}.
$$
The case when $s_1>\frac2p$ or $s_2>\frac1p$ can be similarly
modified as that in Definition \ref{def1.1}.}
\end{defi}

\noindent
{\bf Notations } In all that follows, we shall denote
$$
\fB^0_p \eqdefa \fB^{-1+\frac 2 p,\frac 1 p}_{p}\,,\ \fB^1_p \eqdefa
\fB^{\frac 2 p,\frac 1 p}_{p}\cap\fB^{-1+\frac 2 p,1+\frac 1 p}_{p}
\andf \fB^2_p \eqdefa \fB^{1+\frac 2 p,\frac 1 p}_{p}\cap \fB^{\frac
2 p,1+\frac 1 p}_{p}
$$

Our first result in this paper is as follows:

\begin{thm}
\label{th1}

{\sl Let~$p$ be in~$]3,4[$ and~$r$ in~$[p,6[$. Let us consider an
intial data~$(a_0,u_0)$ in the space~$ \cB_{p}^{\f3p} \times
\fB^0_{p}\cap \cB^{-1+\f3r}_{r}.$
 Then
there exist positive constants $c_0$ and $C_0$ such that if
 \beq
\label{1.6} \eta\eqdefa
\bigl(\mu\|a_0\|_{\cB_{p}^{\f3p}}+\|u_0^h\|_{\fB^0_{p}}\bigr)\exp\Bigl(
\frac {C_0} {\mu^2} \|u_0^3\|_{\fB^0_{p}}^2\Bigr)\leq c_0\mu, \eeq
the system~\eqref{1.3} has a unique global solution
\beq
\label{1.10}
\begin{split}
&a\in\cC_b([0,\infty); \cB_{p}^{\f3p}(\R^3)) \quad\mbox{ and}\quad
u\in\cC_b([0,\infty); \cB^{-1+\f3r}_{r}(\R^3))\cap L^1(\R^+;
\cB^{1+\f3r}_{r}(\R^3)). \end{split} \eeq Moreover, there holds
\beq\label{th1.a}
\begin{split}
&\|u^h\|_{\wt{L}^\infty(\R^+;\fB^0_p)}+\mu\bigl(\|a\|_{\wt{L}^{\infty}(\R^+;\cB_{p}^0)}
+\| u^h\|_{L^1(\R^+;\fB^2_p)}\bigr)\leq C\eta,\\
&\|u^3\|_{\wt{L}^\infty(\R^+;\fB^0_p)}+\mu\|u^3\|_{L^1(\R^+;\fB^2_p)}
\leq 2\|u_0^3\|_{\fB^0_p}+c_2\mu.
\end{split}
 \eeq }

\end{thm}

\begin{rmk}\label{rmk1.2}
 {\sl (1)}\  We
emphasize that for any given function $a, \phi$ in the Schwartz
space $\cS(\R^3)$, any~$p$ in~$]3,4[$, Theorem \ref{th1} implies the
global wellposedness of \eqref{1.3} with initial data of the form
\beq\label{rmk1.2a}
\begin{split}
&a_0^\e(x)= (-\ln \e)^\delta \e^{1+\f1p} a(x_1,x_2, \e x_3)\qquad\mbox{and}\\
&u_0^\epsilon=\e_0(-\ln\e)^{\delta}\
\e^{-(1-\f2p)}\sin\Bigl(\frac{x_1}\e\Bigr)
\bigl(0,-\e\p_3\phi,\p_2\phi\bigr)(x_1,x_2, \e x_3),\end{split} \eeq
 for  $0<\delta<\f12$, and $\e,\e_0$ being sufficiently
 small.  Indeed it is well-known that \beno
 \begin{split}
&\quad\Bigl\|\sin\Bigl(\frac{x_1}{\e}\Bigl)\na\phi(x_1,x_2, x_3)\Bigr\|_{\fB^0_p}\leq
C_\phi
\e^{1-\f2p}, \\
& \|a(x_1,x_2,\e x_3)\|_{\cB^{\f3{p}}_{p}}\leq
C\e^{-\f{1}p}\|a\|_{L^p}^{1-\f3p}\|\na a\|_{L^p}^{\f3p},
\end{split} \eeno which ensures that
$$
\bigl(\mu\|a_0^\e\|_{\cB_{p}^{\f3p}}+\|u_0^{\e,h}\|_{\fB^0_{p}}\bigr)\exp\Bigl(
\frac {C_0} {\mu^2} \|u_0^{\e,3}\|_{\fB^0_{p}}^2\Bigr)
\leq
C\e(-\ln\e)^\d\exp\bigl((-\ln\e)^{2\d}\bigr) \rightarrow 0
$$
which tends to~$0$ when~$\e$ tends to~$0$. Hence
Theorem \ref{th1} implies that \eqref{1.3} with initial data
$(a_0^\e, u_0^\e)$ has a unique global solution $(a^\e, u^\e).$

\no{\sl (2)}\ In the case when $\d=0$ in \eqref{rmk1.2a}, the
homogeneity of the initial density $a_0^\e$ could be much larger. In
fact, it follows from the same line as the proof of part (1) that
\eqref{1.3} with the data \beno
\begin{split}
&a_0^\e(x)=\e^{\f1p} a(x_1,x_2, \e x_3)\qquad\mbox{and}\\
&u_0^\epsilon=\e_0 \e^{-(1-\f2p)}\sin\Bigl(\frac{x_1}\e\Bigr)
\bigl(0,-\e\p_3\phi,\p_2\phi\bigr)(x_1,x_2, \e x_3),\end{split}
\eeno also has a unique global solution for $\e_0,
\|a_0\|_{\cB^{\f3p}_p}$ and $\e$ being sufficiently small.
\end{rmk}

Theorem \ref{th1} also ensures the global wellposedness of
\eqref{1.3} with data of the form: \beno \bigl(a_0(x_h,x_3), (\e
u^h_0(x_h,\e x_3), u_0^3(x_h,\e x_3))\bigr)\eeno for any smooth
divergence free vector field $u_0=(u_0^h,u_0^3)$ and with $\e,$
$\|a_0\|_{\cB^{\frac3p}_p},$ for some~$p$ in~$]3,4[$ being
sufficiently small. Notice that the authors \cite{c-g} proved the
global existence of smooth solutions to 3-D classical Navier-Stokes
system for some large data which are slowly varying in one
direction. The main idea behind the proof in \cite{c-g} is that the
solutions to 3-D Navier-Stokes equations slowly varying in one space
variable can be well approximated by solutions of 2-D Navier-Stokse
equation. Yet just as the classical 2-D Navier-Stokes system, 2-D
inhomogeneous Navier-Stokes equations is also globally wellposed
with general initial data (see \cite{danchin2, LS} for instance).
This motivates us to study the global wellposedness of \eqref{1.3}
with large data which are slowly variable in one direction and which
do not satisfy the nonlinear smallness condition \eqref{1.6}.

\begin{thm}\label{thm5}
{\sl Let $\sigma$ be a real number greater than~$1/4$ and $a_0$ a function  of~$
\cB^{\f3p}_p\cap\cB^{-1+\f3q}_{q}$ for some $p$ in~$]3,4[$ and $q$ in~$
]\f32,2[.$ Let $v_0^h=(v_0^1,v_0^2)$ be a horizontal, smooth
divergence free vector field on $\mathbb{R}^3$, belonging, as well
as all its derivatives, to $L^2(\R_{x_3};\dot{H}^{-1}(\R^2)).$
Furthermore, we assume that  for any $\alpha$ in~$\N^3,$
$\pa^\alpha\pa_3v_0^h$ belongs to~Ê$\fB^{-1,\frac12}_{2}(\R^3)$. Then there
exists a positive $\ep_0$ such that if $\ep \leq \ep_0$,
 the initial data
\begin{equation} \label{1.6cz} a_0^{\ep}(x)=\ep^{\sigma} a_0(x_h,\ep x_3), \quad u_0^{\ep}(x)=(v_0^h(x_h,\ep
x_3), 0)
\end{equation}
 generates a unique global solution
$(a^{\ep},u^{\ep})$ of (\ref{1.3}).}
\end{thm}

\begin{rmk}\label{rmk1.4ag}
(1)\ With $v_0^h$ being  given by Theorem \ref{thm5} and $w_0$ a
smooth divergence free vector field on
 $\mathbb{R}^3,$
I.  Gallagher and the first author proved in\ccite{c-g}  that there exists a positive
$\ep_0$ such that if $0 <\ep \leq \ep_0$, the classical
Navier-Stokes system (which corresponds to $a=0$ in \eqref{1.3})
with the initial data
\begin{equation} \label{cgtrans} u_0^{\ep}(x)=(v_0^h +
\ep w_0^h,w_0^3)(x_h,\ep x_3) \end{equation} has a unique  global
solution.

\no (2)\ G. Gui, J. Huang and and the last author proved in\ccite{GHZ1} similar global
wellposedness result for (\ref{1.3}) with initial data
$a_0^{\ep}(x)=\ep^{\delta_0} a_0(x_h,\ep x_3)$ and initial velocity
given by \eqref{cgtrans} provided that  $a_0 \in W^{1,p}\cap H^2$
for some $p \in (1,2)$ and $\delta_0
>\frac{1}{p}.$ We should point out  that one difficulty in
\cite{GHZ1} is to derive $L^\infty(\R^+; \fB^{1}_p)$ estimate for
the solution $a$ of the free transport equation in \eqref{1.3}.
Toward this, the authors in \cite{GHZ1} assumed more regularities
for $a_0$  and then use an interpolation argument to get this
estimate. The advantage of the argument used in the proof of Theorem
\ref{thm5} is that: as observed from the proof of Theorem \ref{th1},
the isentropic regularities of $a$ is matched with the anisotropic
regularities of $u,$ so that we can still work this problem in the
scaling invariant spaces, which leads to the improvement of the
index $\s>\f12$ in \cite{GHZ1} to be $\s>\f14$ here.

\no (3) It follows from the proof of Theorem \ref{thm5} that we can
prove similar  wellposedness result for \eqref{1.3} with data
$(a_0^\e, u_0^\e)$ given by \eqref{cgtrans} provided that $\e\leq
\e_0$ and $\|a_0^\e\|_{\cB^{\f3p}_p}+\e\|a_0^\e\|_{\cB^{-1+\f3q}_q}$
being sufficiently small and for some $p,q$ satisfying $p$ in~$]3,4[$
and $ q$ in~$]\f32, 2[.$ Nevertheless, as $w_0$ part in
\eqref{cgtrans} satisfies our nonlinear smallness condition
\eqref{1.6}, we choose to investigate the case \eqref{1.6cz} here.
\end{rmk}

The organization of this paper is as follows:

In the second section, we prove some lemmas using Littlewood-Paley theory
 in particular a  lemma of product,  a lemma which explains how to compute the pressure in the case when~$a$ is small in~$\cB_p^{\f3p}$
  and  a lemma of propagation for the transport equation which takes
  into account some anisotropy.

In the third section, we prove Theorem\refer{th1}.

In the forth section, we prove Theorem\refer{thm5}

Let us complete this section by the notations of the paper:

Let $A, B$ be two operators, we denote $[A;B]=AB-BA,$ the commutator
between $A$ and $B$. For $a\lesssim b$, we mean that there is a
uniform constant $C,$ which may be different on different lines,
such that $a\leq Cb$.  We denote by $(a|b)$ the $L^2(\R^3)$ inner
product of $a$ and $b,$ $(d_j)_{j\in\Z}$ (resp.
$(d_{j,k})_{j,k\in\Z^2}$) will be a generic element of $\ell^1(\Z)$
(resp. $\ell^1(\Z^2)$) so that $\sum_{j\in\Z}d_j=1$ (resp.
$\sum_{j,k\in\Z^2}d_{j,k}=1$).

For $X$ a Banach space and $I$ an interval of $\R,$ we denote by
${\cC}(I;\,X)$ the set of continuous functions on $I$ with values in
$X,$  and by ${\mathcal{C}}_b(I;\,X)$ the subset of bounded
functions of ${\mathcal{C}}(I;\,X).$  For $q\in[1,+\infty],$ the
notation $L^q(I;\,X)$ stands for the set of measurable functions on
$I$ with values in $X,$ such that $t\longmapsto\|f(t)\|_{X}$ belongs
to $L^q(I).$

\setcounter{equation}{0}
\section{Some estimates related to Littlewood-Paley analysis}

As we shall frequently use the anisotropic Littlewood-Paley theory,
and in particular aniso\-tropic Bernstein inequalities. For the
convenience of the readers, we first recall the following Bernstein
type lemma from \cite{CZ1, Pa02}:

\begin{lem}\label{lem2.1}
{\sl Let $\cB_{h}$ (resp.~$\cB_{v}$) a ball of~$\R^2_{h}$
(resp.~$\R_{v}$), and~$\cC_{h}$ (resp.~$\cC_{v}$) a ring
of~$\R^2_{h}$ (resp.~$\R_{v}$); let~$1\leq p_2\leq p_1\leq \infty$
and ~$1\leq q_2\leq q_1\leq \infty.$ Then there holds:

\smallbreak\noindent If the support of~$\wh a$ is included
in~$2^k\cB_{h}$, then
\[
\|\partial_{x_h}^\alpha a\|_{L^{p_1}_h(L^{q_1}_v)} \lesssim
2^{k\left(|\al|+2\left(\frac1{p_2}-\frac1{p_1}\right)\right)}
\|a\|_{L^{p_2}_h(L^{q_1}_v)}.
\]
If the support of~$\wh a$ is included in~$2^\ell\cB_{v}$, then
\[
\|\partial_{3}^\beta a\|_{L^{p_1}_h(L^{q_1}_v)} \lesssim
2^{\ell(\beta+(\frac1{q_2}-\frac1{q_1}))} \|
a\|_{L^{p_1}_h(L^{q_2}_v)}.
\]
If the support of~$\wh a$ is included in~$2^k\cC_{h}$, then
\[
\|a\|_{L^{p_1}_h(L^{q_1}_v)} \lesssim 2^{-kN}\sup_{|\al|=N}
\|\partial_{x_h}^\al a\|_{L^{p_1}_h(L^{q_1}_v)}.
\]
If the support of~$\wh a$ is included in~$2^\ell\cC_{v}$, then
\[
\|a\|_{L^{p_1}_h(L^{q_1}_v)} \lesssim 2^{-\ell N} \|\partial_{3}^N
a\|_{L^{p_1}_h(L^{q_1}_v)}.
\]
}
\end{lem}

To consider the product of a distribution in the isentropic Besov
space with a distribution in the anisotropic Besov space, we need
the following result which allows to embed  isotropic Besov spaces
into the anisotropic ones.

 \begin{lem}
 \label{propanisoiso}
{\sl Let~$s$ and~$t$ be positive real numbers. Then for any~$p$ in~$[1,\infty], $ one has
$$
\|f\|_{\fB^{s,t}_{p}} \lesssim \|f\|_{\cB^{s+t}_{p}} \, .
$$
}
\end{lem}
\begin{proof}
Thanks to Definition \ref{def1.2}, one has
$$
\|f\|_{\fB^{s,t}_{p}} = \sum_{j,k\in\Z^2} 2^{js} 2^{kt} \|
\Delta_j^h\Delta_k^v f\|_{L^p} \, .
$$
We separate the above sum into two parts, depending on whether~$k <
j$ or~$ k \geq j$ and we shall only detail the first case (the
second one is identical). We  notice  that if~$k<j$,
$$
\begin{aligned}
 \| \Delta_j^h\Delta_k^v f \|_{L^p}\leq  \sum_{\ell\in\Z}   \|\Delta_\ell \Delta_j^h\Delta_k^v f \|_{L^p}\lesssim
 \sum_{|\ell-j|\leq N_0}\| \Delta_\ell f \|_{L^p}\, .
\end{aligned}
$$
Then we infer from the fact that $t>0$
$$
\begin{aligned}
\sum_{\substack{j\in\Z\\k < j}} 2^{js} 2^{kt} \|
\Delta_j^h\Delta_k^v f\|_{L^p}
&\lesssim \sum_{\substack{j,\ell\in\Z^2\\|j-\ell|\leq N_0}}2^{js}\| \Delta_\ell   f\|_{L^p} \sum_{k< j} 2^{kt}  \\
&\lesssim  \sum_{j\in\Z} 2^{j(s+t)} \| \Delta_j   f\|_{L^p}\lesssim
\|f\|_{\cB^{s+t}_p}.
\end{aligned}
$$
And the result follows.
\end{proof}

In  order to obtain a better description of the regularizing effect
of the transport-diffusion equation, we will use Chemin-Lerner type
spaces $\widetilde{L}^{\lambda}_T(B^s_{p,r}(\R^3))$ (see \cite{BCD}
for instance).

To study product laws between distributions in the anisotropic Besov
spaces, we need to  modify the isotropic para-differential
decomposition of  Bony \cite{Bo} to the setting of anisotropic
version. We first recall the isotropic para-differential
decomposition from \cite{Bo}: let $a$ and~Ê$b$ be in~$ \cS'(\R^3)$,
\beq
\label{pd}\begin{split} &ab=T(a,b)+\cR(a,b), \quad\mbox{or}\quad
ab=T(a,b)+\bar{T}(a,b)+ R(a,b), \quad\hbox{where}\\
& T(a,b)=\sum_{j\in\Z}S_{j-1}a\Delta_jb, \quad
\bar{T}(a,b)=T(b,a),\quad
 \cR(a,b)=\sum_{j\in\Z}\Delta_jaS_{j+2}b, \andf\\
&R(a,b)=\sum_{j\in\Z}\Delta_ja\tilde{\Delta}_{j}b,\quad\hbox{with}\quad
\tilde{\Delta}_{j}b=\sum_{\ell=j-1}^{j+1}\D_\ell a. \end{split} \eeq
In what follows, we shall also use the  anisotropic version of
Bony's decomposition for both horizontal and vertical variables.

As an application of the above basic facts on Littlewood-Paley
theory, we present the following product laws in the anisotropic
Besov spaces.

\begin{lem}
\label{lem2.2}
{\sl Let $p\geq q\geq 1$ with $\f1p+\f1q\leq 1,$ and $s_1\leq
\f2{q},$ $ s_2\leq \f2p$ with $s_1+s_2>0.$ Let $\sigma_1\leq
\f1{q},$ $ \sigma_2\leq \f1{p}$ with $\sigma_1+\sigma_2>0$. Then for
$a$ in~$ \fB^{s_1,\sigma_1}_{q}(\R^3)$ and~$b$ in~$
\fB^{s_2,\sigma_2}_{p}(\R^3)$, the product~$ab$ belongs to~$
\fB^{s_1+s_2-\f2{q},\sigma_1+\sigma_2-\f{1}{q}}_{p}(\R^3),$ and
\beno \|a
b\|_{\fB^{s_1+s_2-\f2{q},\sigma_1+\sigma_2-\f{1}{q}}_{p}}\lesssim
\|a\|_{\fB^{s_1,\sigma_1}_{q}}\|b\|_{\fB^{s_2,\sigma_2}_{p}}. \eeno}
\end{lem}

\begin{proof}  We first get by applying Bony's decomposition \eqref{pd} in both
horizontal and vertical variables that \beq\label{le2.3a}
\begin{split}
ab=&(T^h+\bar{T}^h+R^h)(T^v+\bar{T}^v+R^v)(a,b)\\
=&T^hT^v(a,b)+T^h\bar{T}^v(a,b)+T^hR^v(a,b)+\bar{T}^hT^v(a,b)\\
&+\bar{T}^h\bar{T}^v(a,b)+\bar{T}^hR^v(a,b)+R^hT^v(a,b)+R^h\bar{T}^v(a,b)+R^hR^v(a,b).
\end{split}
\eeq In what follows, we shall  detail the estimates to some typical
terms above, the other cases can be followed along the same line.
Note that $\s_1+\s_2>0,$ we get, by applying Lemma \ref{lem2.1},
that \beno
\begin{split}
\|\D_j^h\D_k^v(T^hR^v(a,b))\|_{L^{p}}\lesssim&
2^{\frac{k}{q}}\sum_{\substack{|j'-j|\leq 4\\ k'\geq
k-N_0}}\|S_{j'-1}^h\D_{k'}^va\|_{L^\infty_h(L^{q}_v)}\|\D_{j'}^h\wt{\D}_{k'}^vb\|_{L^{p}}\\
\lesssim& 2^{\frac{k}{q}}\sum_{\substack{|j'-j|\leq 4\\ k'\geq
k-N_0}}d_{j',k'}2^{-j'(s_1+s_2-\f2q)}2^{-k'(\s_1+\s_2)}\|a\|_{\fB^{s_1,\sigma_1}_{q}}\|b\|_{\fB^{s_2,\sigma_2}_{p}}\\
 \lesssim&
d_{j,k}2^{-j(s_1+s_2-\frac2{q})}2^{-k(\s_1+\s_2-\frac1{q})}\|a\|_{\fB^{s_1,\sigma_1}_{q}}\|b\|_{\fB^{s_2,\sigma_2}_{p}}
.
\end{split}
\eeno The same estimate holds for $T^hT^v(a,b)$ and
$T^h\bar{T}^v(a,b).$

Along the same lines, we obtain
 \beno
\begin{split}
\|\D_j^h\D_k^v(\bar{T}^hR^v(a,b))\|_{L^{p}}\lesssim&2^{2j(\f1q-\f1p)}
2^{\frac{k}{q}}\!\!\sum_{\substack{|j'-j|\leq 4\\ k'\geq
k-N_0}}\|\D_{j'}^h\D_{k'}^va\|_{L^{q}}\|S_{j-1}^h\wt{\D}^v_{k'}b\|_{L^\infty_h(L^{p}_v)}\\
\lesssim&2^{2j(\f1q-\f1p)} 2^{\frac{k}{q}}\!\!\!\!\!\sum_{\substack{|j'-j|\leq
4\\ k'\geq k-N_0}}\!\!\!d_{j',k'}2^{-j'(s_1+s_2-\f2p)}2^{-k'(\s_1+\s_2)}\|a\|_{\fB^{s_1,\sigma_1}_{q}}\|b\|_{\fB^{s_2,\sigma_2}_{p}}\\
\lesssim& d_{j,k}
2^{-j(s_1+s_2-\frac2{q})}2^{-k(\s_1+\s_2-\frac1{q})}\|a\|_{\fB^{s_1,\sigma_1}_{q}}\|b\|_{\fB^{s_2,\sigma_2}_{p}}.
\end{split}
\eeno The same estimate holds for $\bar{T}^hT^v(a,b)$ and
$\bar{T}^h\bar{T}^v(a,b).$
Finally applying Lemma \ref{lem2.1} once again and using the fact
that $s_1+s_2>0, $ $\s_1+\s_2>0,$ gives rise to
\beno
\|\D_j^h\D_k^v(R^hR^v(a,b))\|_{L^{p}} &\lesssim &2^{\frac{2j}{q}}
2^{\frac{k}{q}}\sum_{\substack{j'\geq j-N_0\\
 k'\geq k-N_0}}\|\D_{j'}^h\D_{k'}^va\|_{L^{q}}\|\wt{\D}_{j'}^h\wt{\D}^v_{k'}b\|_{L^{p}}\\
& \lesssim &
2^{\frac{2j}{q}} 2^{\frac{k}{q}}\sum_{\substack{j'\geq
j-N_0\\ k'\geq k-N_0}}d_{j',k'}2^{-j'(s_1+s_2)}2^{-k'(\s_1+\s_2)}\|a\|_{\fB^{s_1,\sigma_1}_{q}}\|b\|_{\fB^{s_2,\sigma_2}_{p}}\\
&\lesssim& d_{j,k}
2^{-j(s_1+s_2-\frac2{q})}2^{-k(\s_1+\s_2-\frac1{q})}\|a\|_{\fB^{s_1,\sigma_1}_{q}}\|b\|_{\fB^{s_2,\sigma_2}_{p}}.
\eeno The same estimate holds for $R^hT^v(a,b)$ and
$R^h\bar{T}^v(a,b).$ This together with \eqref{le2.3a} completes the
proof of Lemma \ref{lem2.2}.
\end{proof}

\medbreak As an application of the laws of product, we state a lemma
which will  describe the way how to compute the pressure in the case
when~$a$ is small.
\begin{lem}
\label{lemmapressureinhomo} {\sl Let $p\in(1,4),$ we consider a
function~$a$ such that~$\|a\|_{\cB^{\f3p}_p}$ is small enough.
If~$\Pi$ satisfies
$$
(D) \qquad \dive((1+a) \nabla \Pi -f)=0
$$
with~$f$ in~$\fB_p^0$, then~$(D)$ has a unique solution which
satisfies
$$
\|\na\Pi\|_{\fB_p^0} \lesssim \|f\|_{\fB_p^0} \quad\hbox{and
thus}\quad\|(1+a)\na \Pi\|_{\fB_p^0} \lesssim \|f\|_{\fB_p^0}.
$$
}
\end{lem}

\begin{proof} We first write ~$(D)$ as
$$
\Delta \Pi = -\dive (a\nabla \Pi) + \dive f.
$$
Applying now the operator~$\nabla \Delta^{-1}$  to this identity implies that
$$
\nabla \Pi = -\cM_a (\nabla \Pi) +\nabla \Delta^{-1} \dive f\with
-\cM_a (g)\eqdefa \nabla \Delta^{-1} \dive (ag).
$$
Laws of product from Lemma\refer{lem2.2} together with
Lemma\refer{propanisoiso} implies that~$\|\cM_a\|_{\cL(\fB^0_p)}
\lesssim \|a\|_{\cB^{\f3p}_p}$ because~$p<4$. Thus,
if~$\|a\|_{\cB^{\f3p}_p}$ is small enough, the
operator~$(\Id-\cM_a)^{-1}$ is well defined as an element
of~$\cL(\fB^0_p)$ by the formula
$$
(\Id-\cM_a)^{-1} = \sum_{k=0}^\infty \cM_a^k.
$$
As~$\nabla\D^{-1} \dive$ is a homogenenous Fourier multiplier of degree~$0$, the lemma is proved.
\end{proof}

Now, we are going the prove a lemma which is a variation about the classical propagation lemma for regularity of index less than~$1$.
\begin{lem}
\label{prop4.1} {\sl Let $a_0$ be in~$\cB_{p}^{\f3p}(\R^3),$ and
$u=(u^h,u^3)$ be a divergence free vector field such that~$\nabla u$ belongs to~$
L^1([0,T], L^\infty(\R^3)).$ Let $f$ be in~$L^1([0,T])$ with $\|\na
u^3(t)\|_{L^\infty}\leq C f(t)$ for all $t$ in~$[0,T].$ We denote
$$
 a_\la\eqdefa
a\exp\Bigl(-\la\int_0^tf(t')\,dt'\Bigr).
$$
 Then, the unique solution~$a$ of
 \beq
  \label{4.1} \p_ta+u\cdot\na a=0,\qquad
a|_{t=0}=a_0
\eeq
 satisfies, for any $t$in~$[0,T]$ and $\la$ large enough,

\beq \label{4.2}
\|a_\la\|_{\wt{L}^{\infty}_t(\cB_{p}^{\f3p})}+\f{\la}2 \int_0^t
f(t') \|a_\la(t') \|_{\cB_{p}^{\f3p}}dt'\leq
\|a_0\|_{\cB_{p}^{\f3p}}+C\|a_\lambda\|_{\wt{L}_t^\infty(\cB_{p}^{\f3p})}
\int_0^t \|\nabla u^h(t') \|_{L^\infty}dt' . \eeq }
\end{lem}
\begin{proof}  The proof of this lemma basically follows from that of Proposition 3.1
in \cite{PZ2}. The novelty of our observation here is that the
$L^1_T(Lip(\R^3))$ estimate of the convection velocity enables us to
propagate the $\cB^{\frac3p}_p$ regularity for \eqref{4.1} when
$p>3.$

 As both the existence and uniqueness of solutions to \eqref{4.1}
essentially follows from the estimate (\ref{4.2}) for some
appropriate approximate solutions to \eqref{4.1}. For simplicity,
here we just present the {\it a priori} estimate \eqref{4.2} for
smooth enough solutions of \eqref{4.1}. In this case, thanks to
\eqref{4.1}, we have \beno \p_ta_\la+\la f(t)a_\la+u\cdot\na
a_\la=0. \eeno Applying $\D_j$ to the above equation and then taking
$L^2$ inner product of the resulting equation with $|\D_j
a_\lam|^{p-2}\D_j a_\lam$, we obtain
\beq
\label{4.3}
 \f1q\f{d}{dt}\|\D_j
a_\la(t)\|_{L^p}^p+\la f(t)\|\D_j a_\la(t)\|_{L^p}^p+\bigl(\D_j
(u\cdot\na a_\la)\ |\ |\D_j a_\lam|^{p-2}\D_ja_\lam\bigr)=0.
\eeq
While as
$\dv u=0,$ we get, by using Bony's decomposition \eqref{pd}, \beno
u\cdot\na a_\la=T(u, \na a_\la)+\cR(u,\na a_\la),\eeno and  a
standard commutator's argument, that  \beno
\begin{split} \bigl(\D_j(T(u,\na a_\la))\ |\ |\D_j
a|^{p-2}\D_ja\bigr)=&\sum_{|j'-j|\leq 5}\Bigl(\bigl([\D_j;
S_{j'-1}u]\D_{j'}\na a_\la \ |\ |\D_j
a_\la|^{p-2}\D_ja_\la\bigr)\\
&+\bigl((S_{j'-1}u-S_{j-1}u)\D_j\D_{j'}\na a_\la \ |\ |\D_j
a_\la|^{p-2}\D_ja_\la\bigr)\Bigr).
\end{split}
\eeno Then we deduce from \eqref{4.3} that
\beq
\label{4.4}
\begin{split} \|\D_ja_\la(t)\|_{L^p}&+\la
\int_0^tf(t')\|\D_ja_\la(t')\|_{L^p}\,dt'\\
&\leq \|\D_ja_0\|_{L^p}+C\Bigl(\sum_{|j'-j|\leq 4}\bigl(\|[\D_j;
S_{j'-1}u]\D_{j'}\na a_\la
\|_{L^1_t(L^p)}\\
&\quad+\|(S_{j'-1}u-S_{j-1}u)\D_j\D_{j'}\na
a_\la\|_{L^1_t(L^p)}\bigr)+\|\cR(u,\na a_\la)\|_{L^1_t(L^p)}\Bigr).
\end{split}
\eeq Applying the classical estimate on commutator (see \cite{BCD}
for instance) leads to \beno
\begin{split}
&\sum_{|j'-j|\leq 4}\|[\D_j; S_{j'-1}u]\D_{j'}\na a_\la
\|_{L^1_t(L^p)}\\
&\ \lesssim  \sum_{|j'-j|\leq 4}\Bigl(\|S_{j'-1}\na
u^h\|_{L^1_t(L^\infty)}\|\D_{j'}a_\la\|_{L^\infty_t(L^p)}+\int_0^t\|S_{j'-1}\na
u^3(t')\|_{L^\infty}\|\D_{j'}a_\la(t')\|_{L^p}\,dt'\Bigr)\\
&\ \lesssim \sum_{|j'-j|\leq 5}\Bigl(d_{j'}2^{-\f{3j'}{p}}\| \nabla
u^h\|_{{L}^1_t(L^\infty)}\|a_\la\|_{\wt{L}^\infty_t(\cB_{p}^{\f3p})}+\int_0^t\|
\nabla u^3(t')\|_{L^\infty}\|\D_{j'}a_\la(t')\|_{L^p}\,dt'\Bigr)\\
&\ \lesssim d_j2^{-\f{3j}{p}}\bigl(\| \nabla
u^h\|_{{L}^1_t(L^\infty)}\|a_\la\|_{\wt{L}^\infty_t(\cB_{p}^{\f3p})}
+\int_0^tf(t')\|a_\la(t')\|_{\cB_{p}^{\f3p}}\,dt'\bigr).
\end{split}
\eeno Similarly we get, by applying Lemma \ref{lem2.1}, that \beno
&&\sum_{|j'-j|\leq 4}\|(S_{j'-1}u-S_{j-1}u)\D_j\D_{j'}\na
a_\la\|_{L^1_t(L^p)}\\
&&\lesssim  \sum_{|j'-j|\leq 4}\Bigl(\|(S_{j'-1}\na u^h-S_{j-1}\na
u^h)\|_{L^1_t(L^\infty)}\|\D_j
a_\la\|_{L^\infty_t(L^p)}\\
&&\qquad+\int_0^t\|(S_{j'-1}\na u^3-S_{j-1}\na
u^3)(t')\|_{L^\infty}\|\D_j
a_\la(t')\|_{L^p}\,dt'\Bigr)\\
&& \lesssim d_j2^{-\f{3j}{p}}\| \nabla
u^h\|_{{L}^1_t(L^{\infty})}\|a_\la\|_{\wt{L}^\infty_t(\cB_{p}^{\f3p})}
+\sum_{|j'-j|\leq 4}\int_0^t\|\nabla
u^3(t')\|_{L^{\infty}}\|\D_j a_\la(t')\|_{L^p}\,dt' \\
&&\lesssim d_j2^{-\f{3j}{p}}\bigl(\| \nabla
u^h\|_{{L}^1_t(L^{\infty})}\|a_\la\|_{\wt{L}^\infty_t(\cB_{p}^{\f3p})}
+\int_0^tf(t')\|a_\la(t')\|_{\cB_{p}^{\f3p}}\,dt'\bigr). \eeno

On the other hand, as  $p>3$ and $\nabla a\in
\wt{L}_T^\infty(\cB^{\f3p-1}_{p})$, applying Lemma \ref{lem2.1} once
again gives rise to \beno
\begin{split}
\|S_{j'+2}\na_ha_\la\|_{L^\infty_t(L^p)}\lesssim &\sum_{\ell\leq
j'-2}2^{\ell}\|\D_\ell a_\la\|_{L^\infty_t(L^p)}\\
\lesssim&\sum_{\ell\leq j'-2}d_{\ell} 2^{\ell
(1-\f3p)}\|a_\la\|_{\wt{L}^\infty_t(\cB_{p}^{\f3p})}\lesssim&
d_{j'}2^{j'(1-\frac{3}{p})}\|a_\la\|_{\wt{L}^\infty_t(\cB_{p}^{\f3p})},\end{split}
\eeno so that \beno
\begin{split}
\sum_{j'\geq
j-N_0}\|S_{j'+2}\na_ha_\la\|_{L^\infty_t(L^p)}\|\D_{j'}u^h\|_{L^1_t(L^\infty)}
\lesssim & \sum_{j'\geq j-N_0} d_{j'}2^{-\f{3j'}p}
\|a_\la\|_{\wt{L}^\infty_t(\cB_{p}^{\f3p})}\|\na
u^h\|_{L^1_t(L^\infty)}\\
\lesssim & \,\,d_j2^{-\f{3j}p}\|\na
u^h\|_{L^1_t(L^\infty)}\|a_\la\|_{\wt{L}^\infty_t(\cB_{p}^{\f3p})}.
\end{split}
\eeno It follows from the same lines that \beno
\begin{split}
\sum_{j'\geq j-N_0}&\int_0^t\|S_{j'+2}\p_3a_\la(t')\|_{L^p}\|\D_{j'}u^3(t')\|_{L^\infty}\,dt'\\
\lesssim &\sum_{j'\geq
j-N_0}2^{j'(1-\frac3p)}\int_0^td_{j'}(t')\|a_\la(t')\|_{\cB^{\f3p}_p}
\|\D_{j'} u^3(t')\|_{L^\infty}\,dt'\\
\lesssim &\sum_{j'\geq j-N_0}d_{j'}2^{-\f{3j'}p}
\int_0^t\|a_\la(t')\|_{\cB^{\f3p}_p} \|\nabla
u^3(t')\|_{L^\infty}\,dt' \lesssim  d_j
2^{-\f{3j}{p}}\int_0^tf(t')\|a_\la(t')\|_{\cB_{p}^{\f3p}}\,dt'.
\end{split}
\eeno As a consequence, we obtain
 \beno
\|\D_j(\cR(u,{\na a_\la}))\|_{L^1_t(L^p)}
&\lesssim &
\sum_{j'\geq
j-N_0}\Bigl(\|S_{j'+2}\na_ha_\la\|_{L^\infty_t(L^p)}\|\D_{j'}u^h\|_{L^1_t(L^\infty)}\\
 & &\quad\qquad{}+\int_0^t\|S_{j'+2}\p_3a_\la(t')\|_{L^p}\|\D_{j'}u^3(t')\|_{L^\infty}\,dt'\Bigr)\\
& \lesssim& d_j2^{-\f{3j}{p}}\Bigl(\| \nabla
u^h\|_{{L}^1_t(L^{\infty})}\|a_\la\|_{\wt{L}^\infty_t(\cB_{p}^{\f3p})}
+\int_0^tf(t')\|a_\la(t')\|_{\cB_{p}^{\f3p}}\,dt'\Bigr).
\eeno Substituting the above estimates into \eqref{4.4}
and taking summation for $j$ in~$\Z,$ we arrive at \beno
\begin{split}
\begin{split}
\|a\|_{\wt{L}^{\infty}_t(\cB_{p}^{\f3p})}+&\la\int_0^tf(t')\|a_\la(t')\|_{\cB_{p}^{\f3p}}\,dt'\\
&\leq \|a_0\|_{\cB_{p}^{\f3p}}+C\Bigl(\|\nabla
u^h\|_{L^1_t(L^{\infty})}\|a\|_{\wt{L}^{\infty}_t(\cB_{p}^{\f3p})}+\int_0^tf(t')\|a_\la(t')\|_{\cB_{p}^{\f3p}}\,dt'
\Bigr).
\end{split}
\end{split}
\eeno Taking $\la\geq 2C$ in the above inequality, we conclude the
proof of \eqref{4.2}.
\end{proof}

Following the same line to the proof of Lemma \ref{prop4.1}, we can
also prove the following Lemma, which will be used in the proof of
Theorem \ref{thm5} in Section \ref{sectcz}.

\begin{lem}\label{lem3.2cz}
{\sl Let $q$ be in~$ [1,\infty]$ and $s$ in~$ ]0,1[.$ Then given an initial data
$a_0$ in~$\cB^s_q(\R^3)$ and  a vector field~$u$ in~$ L^1([0,T]; Lip(\R^3))$ with $\dv
u=0, $ \eqref{4.1} has a unique solution $a$ in~$
\cC([0,T];\cB^s_q(\R^3))\cap \wt{L}^\infty_T(\cB^{s}_q(\R^3)).$
Moreover, there holds \beq\label{2.6cz}
\|a\|_{\wt{L}^\infty_T(\cB^{s}_{q})}\leq
\|a_0\|_{\cB^s_q}\exp\Bigl(C\int_0^T\|\na
u(t')\|_{L^\infty}\,dt'\Bigr). \eeq }
\end{lem}

\begin{proof} Notice once again that the proof of Lemma \ref{lem3.2cz} basically follows from \eqref{2.6cz},
 we shall only detail the proof of
\eqref{2.6cz} for smooth enough solutions to \eqref{4.1}. Indeed it
follows from the  proof of \eqref{4.4} that \beno\begin{split}
\frac{d}{dt}\|\D_ja(t)\|_{L^q}\leq &\sum_{|j'-j|\leq 4}\|[\D_j;
S_{j'-1}u]\na \D_{j'}a(t)\|_{L^q}\\
&+\sum_{|j'-j|\leq 4}\|\bigl(S_{j'-1}u-S_{j-1}u\bigr)\cdot\na
\D_{j'}\D_ja(t)\|_{L^q}\\
&+\sum_{j'\geq j-N_0}\|\D_{j'}uS_{j'+2}\na a(t)\|_{L^q}.
\end{split}
\eeno We get by using the classical commutator's estimate (see
\cite{BCD} for instance) that
\beno
\sum_{|j'-j|\leq 4}\|[\D_j; S_{j'-1}u]\na
\D_{j'}a(t)\|_{L^q} & \lesssim & \sum_{|j'-j|\leq 4} \|S_{j'-1}(\na
u)(t)\|_{L^\infty}\|\D_{j'}a(t)\|_{L^q}\\
 & \lesssim & d_j(t)2^{-js}\|\na
u(t)\|_{L^\infty}\|a(t)\|_{\cB^{s}_{q}}.
\eeno
The same estimate holds for $\sum_{|j'-j|\leq
4}\|\bigl(S_{j'-1}u-S_{j-1}u\bigr)\na \D_{j'}\D_ja(t)\|_{L^q}.$
Whereas  applying Bernstein's Lemma and using the fact that $s<1$
yields \beno
\begin{split}
\sum_{j'\geq j-N_0}\|\D_{j'}uS_{j'+2}\na a(t)\|_{L^q}\lesssim &
\sum_{j'\geq j-N_0}\|\D_{j'}u(t)\|_{L^\infty}\|S_{j'+2}\na
a(t)\|_{L^q}\\
\lesssim &d_j(t)2^{-js}\|\na
u(t)\|_{L^\infty}\|a(t)\|_{\cB^{s}_{q}}.
\end{split}
\eeno As a consequence, we arrive at \beno
\|\D_ja\|_{L^\infty_t(L^q)}\leq\|\D_ja_{0}\|_{L^q}+2^{-js}\int_0^td_j(t')
\|\na u(t')\|_{L^\infty}\|a(t')\|_{\cB^{s}_{q}}\,dt',\eeno which
gives rise to \beno \|a\|_{\wt{L}^\infty_t(\cB^{s}_{q})}\leq
\|a_{0}\|_{B^{s}_{q}}+C\int_0^t\|\na
u(t')\|_{L^\infty}\|a(t')\|_{B^{s}_{q}}\,dt'. \eeno Applying
Gronwall inequality  leads to \eqref{2.6cz}.
\end{proof}

\setcounter{equation}{0}\label{sect3}
\section{The proof of Theorem\refer{th1}}
We shall only prove that if~$(a_0,u_0)$ is a smooth initial data
satisfying the smallness condition\refeq {1.6}  then the associated
solution of~$(a,u)$ of \eqref{1.3} satisfies \eqref{th1.a}, which
implies a global control of the~$L^1$ in time with value
in~$L^\infty$ for the gradient of~$u$. With this estimate, it is
standard to prove the $\wt{L}^\infty(\R^+;\cB^{-1+\f3r}_r(\R^3))\cap
L^1(\R^+;\cB^{1+\f3r}_r(\R^3))$ for the velocity field (see
\cite{abidi, AP, PZ2} for instance). In order to prove the existence
part of Theorem\refer{th1}, we regularize the initial data and then
pass to the limit. These technical details are omitted. The
uniqueness part of  Theorem\refer{th1} follows from Theorem 1 of
\cite{DM}.

 Let us denote by~$T^\star$ the maximal time of existence of the solution $(a,u)$ of \eqref{1.3} associated
 with the smooth initial data~$(a_0,u_0)$. Let us consider~$T^+$  defined by
\beq
\label{th1demoeq0}
 T^+ \eqdefa \sup\Bigl\{ T<T^\star\,/\ \zeta_T\eqdefa
 \mu \|a\|_{L^\infty_T(\cB^{\f3p}_p)} +\|u^h\|_{L^\infty_T(\fB^0_p)} +\mu\|u^h\|_{L^1_T(\fB^2_p)} \leq c_0\mu\Bigr\}.
\eeq
  where~$c_0$ will be chosen small enough later on.

  We want first to estimate~$\|g\|_{L^1_T(\fB^0_p)}$ where
  $$
  g(a,u)\eqdefa -u\cdot\nabla u +\mu a \D u-(1+a)\nabla \Pi
  $$
 As in Lemma\refer{prop4.1}, we  define
\beq \label{p.1} b_\lam(t) \eqdefa b(t) \exp\Bigl(-\lam \int_0^t
\|u^3(t',\cdot)\|_{\fB^2_p} dt'\Bigr). \eeq This will allow to make
the term~$\mu a_\lam \Delta u^3$ integrable thanks to
Lemma\refer{prop4.1}.  Notice that taking space divergence to the
momentum equation of \eqref{1.3} gives $\dive g(a,u)=0,$
Lemma\refer{lemmapressureinhomo} implies that
$$
\|g(a,u)_\lam(t)  \|_{\fB^0_p} \lesssim \|(\dive (u\otimes u)-\mu a\D u)_\lam\|_{\fB^0_p}.
$$
Let us estimate the righthand side term. The key point to the
estimation is that it does not contain any  terms  which are
quadratic with respect  to~$u^3$.

 If~$(j,k)$ is in~$\{1,2\}^2$, we have, thanks to law of product of Lemma\refer{lem2.2},
 \ben
 \|\partial_j(u^ju^k)_\lam \|_{\fB^0_p} &=  &  \|(\partial_j(u^ju^k_\lam) \|_{\fB^0_p}\nonumber\\
\label{th1demoeq1} & \lesssim & \|u^h\|_{\fB^0_p} \|u^h_\lam\|_{\fB^2_p}  +  \|u^h_\lam\|_{\fB^0_p} \|u^h\|_{\fB^2_p}.
 \een
Because~Ê$p<4$,  law of product of Lemma\refer{lem2.2} and $\dive
u=0$ implies that
  \ben
 \|(\partial_3(u^3u^k)_\lam\|_{\fB^0_p} &=  &  \|\partial_3u^3u^k_\lam+u^3\p_3u^k_\la\|_{\fB^0_p}\nonumber\\
\label{th1demoeq2} & \lesssim &
\|u^h\|_{\fB^2_p}\|u^h_\lam\|_{\fB^0_p} +\|u^3\|_{\fB^1_p}
\|u^h_\lam\|_{\fB^1_p}.
 \een
 The term~$\partial_3(u^3)^2$, which is the only possible quadratic term,
  is equal to~$-2 u^3\dive_h u^h$ thanks to divergence free condition. As above, we have
  \ben
 \|(\partial_3(u^3)^2)_\lam \|_{\fB^0_p} &=  &  2\|u^3\dive_hu^h_\lam\|_{\fB^0_p}\nonumber\\
\label{th1demoeq3} & \lesssim & \|u^3\|_{\fB^1_p} \|u^h_\lam\|_{\fB^1_p}.
 \een
 Laws of product of Lemma\refer{lem2.2} together with Lemma\refer{propanisoiso} gives
 \ben
\label{th1demoeq4}
 \mu \|(a\D u^h)_\lam \|_{\fB^0_p}   \lesssim   \mu \|a\|_{\cB^{\f3p}_p} \|u^h_\lam\|_{\fB^2_p}\andf
 \mu \|(a\D u^3)_\lam \|_{\fB^0_p}  \lesssim   \mu \|a_\lam\|_{\cB^{\f3p}_p} \|u^3\|_{\fB^2_p}.
 \een
 Lemma\refer{lemmapressureinhomo}  and Estimates\refeq{th1demoeq1}--(\ref{th1demoeq4})  gives, for any positive~$\lam$,
 \beq
 \label{th1demoeq6}
\begin{split}  \|g(a,u)_\lam(t)  \|_{\fB^0_p} & \lesssim  \|u^h\|_{\fB^0_p} \|u^h_\lam\|_{\fB^2_p}  +  \|u^h_\lam\|_{\fB^0_p} \|u^h\|_{\fB^2_p}+\|u^3\|_{\fB^1_p} \|u^h_\lam\|_{\fB^1_p}\\
&\qquad\qquad\qquad\qquad\qquad\qquad\qquad{}+\mu
\|a\|_{\cB^{\f3p}_p} \|u^h_\lam\|_{\fB^2_p}+\mu
\|a_\lam\|_{\cB^{\f3p}_p} \|u^3\|_{\fB^2_p}.
\end{split}
 \eeq
 Let us first estimate~$u^3$. As~$u^3$ satisfies
  $$
 \partial_t u^3 -\D u^3  = \bigl(-u\cdot \nabla u +\mu a\D u +(1+a)\nabla \Pi\bigr)^3,
 $$
 we get, by using\refeq{th1demoeq6} with~$\lam=0$, that
\beno
\begin{split}
& 2^{j\left (-1+\frac 2p\right)}2^{\frac{k}p} \bigl( \|\D_j^h\D_k^v
u^3\|_{L^\infty_T(L^p)}  +\mu(2^{2k}+2^{2j} \bigr)
 \|\D_j^h\D_k^v u\|_{L^1_T(L^p)} \bigr) \\
&\quad{}
\lesssim 2^{j\left (-1+\frac 2p\right)}2^{\frac{k}p}\|\D_j^h\D_k^v u^3_0\|_{L^p}+\int_0^T d_{j,k}(t) \bigl (  \|u^h(t)\|_{\fB^0_p} \|u^h(t)\|_{\fB^2_p}+\|u^3(t)\|_{\fB^1_p} \|u^h(t)\|_{\fB^1_p}\\
&\qquad\qquad\qquad\qquad\qquad\qquad\qquad\qquad\qquad\qquad{}+\mu
\|a(t)\|_{\cB^{\f3p}_p}( \|u^h(t)\|_{\fB^2_p}+
\|u^3(t)\|_{\fB^2_p})\bigr)dt.
\end{split}
\eeno
After summation,  this gives
\beno
\begin{split}
&
\| u^3\|_{\wt L^\infty_T(\fB^0_p)}  +\mu \|\ u^3\|_{L^1_T(\fB^2_p)}  \lesssim \|u^3_0\|_{\fB^0_p}+ \int_0^T  \bigl (\|u^h(t)\|_{\fB^0_p} \|u^h(t)\|_{\fB^2_p} \\
&\qquad\qquad\qquad\qquad\qquad{} +\|u^3(t)\|_{\fB^1_p}
\|u^h(t)\|_{\fB^1_p}+\mu \|a(t)\|_{\cB^{\f3p}_p}
(\|u^h(t)\|_{\fB^2_p}+ \|u^3(t)\|_{\fB^2_p})\bigr)dt.
\end{split}
\eeno
By interpolation, we have
$$
\|u^3(t)\|_{\fB^1_p} \|u^h(t)\|_{\fB^1_p} \leq \|u^3(t)\|^\frac 1 2_{\fB^0_p}\|u^3(t)\|_{\fB^2_p}^{\frac 1 2}
 \|u^h(t)\|_{\fB^0_p}^{\frac 1 2}| \|u^h(t)\|_{\fB^2_p}^{\frac 1 2}.
$$
Using the induction hypothesis\refeq{th1demoeq0} and  Cauchy
-Schwartz inequality, we get \beno \begin{split} \| u^3\|_{\wt
L^\infty_T(\fB^0_p)} +\mu \|\ u^3\|_{L^1_T(\fB^2_p)} \lesssim
\|u^3_0\|_{\fB^0_p} +\frac {\zeta_T^2}{\mu}& +\zeta_T
\|u^3\|_{L^1_T(\fB^2_p)}\\
& +\frac{\zeta_T}{ \mu} \bigl(\|u^3\|_{L^\infty_T(\fB^0_p)} \mu
\|u^3\|_{L^1_T(\fB^2_p)} \bigr)^{\frac 1 2}.
\end{split}
\eeno Thus, if~$c_0$  is small enough in \eqref{th1demoeq0}, we get
 \beq
 \label{th1demoeq7}
\forall T<T^\star,\ \| u^3\|_{\wt L^\infty_T(\fB^0_p)}  +\mu \|\ u^3\|_{L^1_T(\fB^2_p)}
\lesssim \|u^3_0\|_{\fB^0_p} +\zeta_T.
 \eeq
 The estimate on~$u^h$ is different. Because of the term~$\mu a\D u^3$ which has no chance to be small
 and which appears  in the equation of~$u^h$, we need to use conjugating with an exponential weight. Let us point out that~$u_\lam$ is the solution of
 $$
 \left\{
\begin{array}{c}
\partial_t u_\lam -\mu\D u_\lam +\lam \|u^3(t)\|_{\fB^2_p} u_\lam =\bigl (-u\cdot\nabla u+\mu a \D u -(1+a)\Pi\bigr)_\lam,\\
\dive u_\lam =0\,,\ u_{|t=0} =0.
\end{array}
\right.
 $$
 Let us consider any subinterval~$I=[I^-,I^+]$ of~$[0,T]$.  Then applying\refeq{th1demoeq6}, we infer
 \beno
\begin{split}
&
\| u^h_\lam\|_{\wt L^\infty(I;\fB^0_p)}  +\mu \|\ u^h_\lam\|_{L^1(I;\fB^2_p)}   \lesssim \|u^h_\la(I^-)\|_{\fB^0_p}+ \int_I \bigl (\|u^h_\lam(t)\|_{\fB^0_p} \|u^h(t)\|_{\fB^2_p}\\
&\qquad\qquad\qquad\qquad\qquad\qquad\qquad{} +\|u^h(t)\|_{\fB^0_p}
\|u^h_\lam(t)\|_{\fB^2_p}
+\|u^3(t)\|_{\fB^1_p} \|u^h_\lam(t)\|_{\fB^1_p}\\
&\qquad\qquad\qquad\qquad\qquad\qquad\qquad\qquad{}+\mu
\|a(t)\|_{\cB^{\f3p}_p} \|u^h_\lam(t)\|_{\fB^2_p}+\mu
\|a_\lam(t)\|_{\cB^{\f3p}_p} \|u^3(t)\|_{\fB^2_p}\bigr)dt.
\end{split}
\eeno Using the induction hypothesis \eqref{th1demoeq0} and
Cauchy-Schwarz inequality, this gives
$$
\longformule{ \| u^h_\lam\|_{\wt L^\infty(I;\fB^0_p)}  +\mu \|\
u^h_\lam\|_{L^1(I;\fB^2_p)}   \lesssim \|u^h_\lam(I^-)\|_{\fB^0_p}+
\frac{\zeta_T}\mu  \bigl (\|u^h_\lam\|_{L^\infty(I;\fB^0_p)} +\mu
\|u^h_\lam\|_{L^1(I;\fB^2_p)}\bigr) } { {}
 +\|u^3\|_{L^2(I;\fB^1_p)} \|u^h_\lam\|_{L^2(I;\fB^1_p)}+\mu \int_I \|a_\lam(t)\|_{\cB^{\f3p}_p} \|u^3(t)\|_{\fB^2_p}dt.
 }
 $$
 By interpolation, this gives
 $$
 \longformule{
\| u^h_\lam\|_{\wt L^\infty(I;\fB^0_p)}  +\mu \|\
u^h_\lam\|_{L^1(I;\fB^2_p)} \lesssim \|u^h_\lam(I^-)\|_{\fB^0_p}+
\mu \int_I \|a_\lam(t)\|_{\cB^{\f3p}_p} \|u^3(t)\|_{\fB^2_p}dt
 }
 {{}
+\Bigl( \frac{\zeta_T}{\mu} +\frac 1 {\mu^{\frac 1 2}}
\|u^3\|_{L^2(I;\fB^1_p)}\Bigr) \bigl
(\|u^h_\lam\|_{L^\infty(I;\fB^0_p)} +\mu
\|u^h_\lam\|_{L^1(I;\fB^2_p)}\bigr). }
$$
The induction hypothesis \eqref{th1demoeq0} implies that, if~$c_0$
is chosen small enough in \eqref{th1demoeq0}, then
$$
\longformule{
\| u^h_\lam\|_{\wt L^\infty(I;\fB^0_p)}  +\mu \|\ u^h_\lam\|_{L^1(I;\fB^2_p)}
 \lesssim \|u^h_\lam(I^-)\|_{\fB^{\f3p}_p}+\mu \int_I \|a_\lam(t)\|_{\cB^{\f3p}_p} \|u^3(t)\|_{\fB^2_p}dt
}
{{}
+\frac 1 {\mu^{\frac 1 2}} \|u^3\|_{L^2(I;\fB^1_p)} \bigl (\|u^h_\lam\|_{L^\infty(I;\fB^0_p)} +\mu  \|u^h_\lam\|_{L^1(I;\fB^2_p)}\bigr).
}
$$
Thus, two  constant~$C_0$ and~$C_1$ exist such that, if the
interval~Ê$I$ satisfies \beq \label{th1demoeq8} \int_I
\|u^3(t)\|^2_{\fB^1_p} dt \leq \frac \mu {C_1}\,\virgp \eeq then we
have
 \beq
\label{th1demoeq9} \| u^h_\lam\|_{\wt L^\infty(I;\fB^0_p)}  +\mu \|\
u^h_\lam\|_{L^1(I;\fB^2_p)}   \leq C_0
\Bigl(\|u^h_\lam(I^-)\|_{\fB^0_p}+ \mu \int_I
\|a_\lam(t)\|_{\cB^{\f3p}_p} \|u^3(t)\|_{\fB^2_p}dt\Bigr). \eeq Now
let us decompose the interval~Ê$[0,T]$ into intervals such that the
smallness condition\refeq{th1demoeq8} is satisfied. Let us define
the sequence~$(t_j)_{0\leq j\leq N}$ such that~$t_0=0$,~$t_N=T$,
$$
\forall j\in \{0,\cdots, N-2\}\,,\ \int_{t_j}^{t_{j+1} }
\|u^3(t)\|_{\fB^1_p}^2 dt =\frac \mu {C_1}\andf \int_{t_{N-1}}
^{t_N} \|u^3(t)\|_{\fB^1_p}^2 dt \leq \frac \mu {C_1}\,\cdotp
$$
Let us observe that
$$
\int_0^T  \|u^3(t)\|_{\fB^1_p}^2 dt \geq \frac \mu {C_1} (N-2)
$$
which implies that the number of intervals~$N$ satisfies \beq
\label{th1demoeq10} N\leq \frac {C_1} \mu \int_0^T
\|u^3(t)\|_{\fB^1_p}^2 dt  +2. \eeq Now let us prove by induction
that, for any~$j\leq N$, we have
$$
(P_j)\qquad \|u^h_\lam \|_{L^\infty([0,t_j],\fB^0_p)} +\mu
\|u^h_\lam \|_{L^1([0,t_j],\fB^2_p)} \leq C_0^j
\Bigl(\|u_0\|_{\fB^0_p} +\mu \int_0^{t_j} \|a_\lam(t)
\|_{\cB^{\f3p}_p} \|u^3(t)\|_{\fB^0_p} dt\Bigr).
$$
For~$j=1$, it is simply\refeq{th1demoeq9} applied with~$I=[0;t_1]$.
Now, let us assume~$(P_j)$ for~$j\leq N-1$. Applying\refeq{th1demoeq9} with~$I=[t_j,t_{j+1}]$ gives
$$
\| u^h_\lam\|_{\wt L^\infty([t_j,t_{j+1}]; \fB^0_p)} +\mu \|u^h_\lam\|_{L^1([t_j,t_{j+1}];\fB^2_p)} \leq C_0
\Bigl(\|u^h_\lam(t_j)\|_{\fB^0_p}+\mu \int_{t_j}^{t_{j+1}}\!
\|a_\lam(t)\|_{\cB^{\f3p}_p} \|u^3(t)\|_{\fB^2_p}dt\Bigr).
$$
The induction hypothesis~$(P_j)$ implies that
$$
\longformule
{ \| u^h_\lam\|_{\wt L^\infty([t_j,t_{j+1}]; \fB^0_p)}
+\mu \|u^h_\lam\|_{L^1([t_j,t_{j+1}];\fB^2_p)}
}
 {{} \leq C_0^{j+1}
\Bigl(\|u_0\|_{\fB^0_p} +\mu \int_0^{t_j} \|a_\lam(t)
\|_{\cB^{\f3p}_p} \|u^3(t)\|_{\fB^0_p} dt\Bigr) +C_0 \mu
\int_{t_j}^{t_{j+1}}\! \|a_\lam(t)\|_{\cB^{\f3p}_p}
\|u^3(t)\|_{\fB^2_p}dt }
$$
which gives obviously~$(P_{j+1})$ and thus~$(P_N)$. Because
of\refeq{th1demoeq10}, this gives \beq\label{p.1w} \begin{split}\|
u^h_\lam\|_{\wt L^\infty_T(\fB^0_p)}  &+\mu \|\
u^h_\lam\|_{L^1_T(\fB^2_p)} \\
& \leq C_0 ^{\frac {C_1} \mu\int_0^T \|u^3(t)\|_{\fB^1_p}^2dt + 2}
\Bigl(\|u^h_0\|_{\fB^0_p} +\mu \int_0^{T} \|a_\lam(t)
\|_{\cB^{\f3p}_p} \|u^3(t)\|_{\fB^2_p} dt\Bigr). \end{split} \eeq On
the other hand, notice from Lemma \ref{lem2.1} that $\|\na
u^3(t)\|_{L^\infty}\leq C\|u^3(t)\|_{\fB^2_p},$ for $\lam$ large
enough, we get, by applying Lemma\refer{prop4.1} with
$f(t)=\|u^3(t)\|_{\fB^2_p},$ that
\beno
\|a_\la\|_{\wt{L}^{\infty}_T(\cB_{p}^{\f3p})}+\f{\la}2 \int_0^T
\|u^3(t)\|_{\fB^2_p}\|a_\la(t) \|_{\cB_{p}^{\f3p}}dt\leq
\|a_0\|_{\cB_{p}^{\f3p}}+C\|u^h\|_{L^1_T(\fB^2_p)}\|a_\lambda\|_{\wt{L}_t^\infty(\cB_{p}^{\f3p})}.
\eeno
Thus as long as $\zeta_T/\mu$ is chosen sufficiently small,
 the induction hypothesis \eqref{th1demoeq0} leads to
\beq
\label{p.1q}
\forall T<T^\star\,,\ \|a_\la\|_{\wt{L}^{\infty}_T(\cB_{p}^{\f3p})}+{\la}
\int_0^T \|u^3(t)\|_{\fB^2_p}\|a_\la(t) \|_{\cB_{p}^{\f3p}}dt\leq
2\|a_0\|_{\cB_{p}^{\f3p}}. \eeq Substituting \eqref{p.1q} into
\eqref{p.1w} gives rise to \beno \begin{split} \| u^h_\lam\|_{\wt
L^\infty_T(\fB^0_p)}
+\mu(\|a_\la\|_{\wt{L}^{\infty}_T(\cB_{p}^{\f3p})}&+ \|\
u^h_\lam\|_{L^1_T(\fB^2_p)})\\
&  \lesssim \bigl(\|u^h_0\|_{\fB^0_p} +\mu
\|a_0\|_{\cB^{\f3p}_p}\bigr) \exp \Bigl(\frac {C_1'} \mu \int_0^T
\|u^3(t)\|_{\fB^1_p}^2dt \Bigr)
\end{split}
\eeno
for $\lam$ large enough and $T\leq T^+.$ While thanks to
\eqref{p.1}, one has \beno \bigl(\| u^h\|_{\wt L^\infty_T(\fB^0_p)}
+\mu \|\
u^h\|_{L^1_T(\fB^2_p)}\bigr)\exp\Bigl(-\la\int_0^T\|u^3(t)\|_{\fB^2_p}dt\Bigr)
\leq \| u^h_\lam\|_{\wt L^\infty_T(\fB^0_p)}  +\mu \|\
u^h_\lam\|_{L^1_T(\fB^2_p)}. \eeno As a consequence, we obtain \beno
\begin{split}
\| u^h\|_{\wt L^\infty_T(\fB^0_p)}
+&\mu(\|a\|_{\wt{L}^{\infty}_T(\cB_{p}^{\f3p})}+ \|\
u^h\|_{L^1_T(\fB^2_p)})\\ \lesssim& \bigl(\|u^h_0\|_{\fB^0_p} +\mu
\|a_0\|_{\cB^{\f3p}_p}\bigr) \exp \Bigl(C_1''\int_0^T\bigl(\f1\mu
\|u^3(t)\|_{\fB^1_p}^2+\|u^3(t)\|_{\fB^2_p}\bigr)dt \Bigr),
\end{split}
\eeno for some sufficiently large constant $C_1''.$ This together
with \eqref{th1demoeq7} implies that \beq \label{p.1t} \| u^h\|_{\wt
L^\infty_T(\fB^0_p)}+\mu
(\|a\|_{\wt{L}^{\infty}_T(\cB_{p}^{\f3p})}+\|\
u^h\|_{L^1_T(\fB^2_p)})\leq C_2\bigl(\|u^h_0\|_{\fB^0_p} +\mu
\|a_0\|_{\cB^{\f3p}_p}\bigr)
\exp\Bigl(\f{C_0}{\mu^2}\|u_0^3\|_{\fB^0_p}\Bigr) \eeq  for $t\leq
T^+,$ provided that $\zeta_T/\mu\leq c_0$ is small enough.

We now claim that $T^+=\infty$ if the initial data $(a_0,u_0)$
satisfies \eqref{1.6}. Otherwise, we infer from \eqref{p.1t} that
\beno
 \| u^h\|_{\wt
L^\infty_T(\fB^0_p)}+\mu
(\|a\|_{\wt{L}^{\infty}_T(\cB_{p}^{\f3p})}+\|\
u^h\|_{L^1_T(\fB^2_p)})\leq C_2\eta\quad\mbox{for}\quad t\leq T^+.
\eeno In particular if we choose $\eta$ in \eqref{1.6} is so small
that $\eta\leq \f{c_0}{2C_2},$ one has \beno
 \| u^h\|_{\wt
L^\infty_T(\fB^0_p)}+\mu
(\|a\|_{\wt{L}^{\infty}_T(\cB_{p}^{\f3p})}+\|\
u^h\|_{L^1_T(\fB^2_p)})\leq \f{c_0}2\quad\mbox{for}\quad t\leq T^+,
\eeno which contradicts with   the induction hypothesis
\eqref{th1demoeq0}, and which in turn shows that $T^+=\infty$ under
the assumption \eqref{1.6}. Furthermore, \eqref{th1demoeq7} and
\eqref{p.1t} ensures \eqref{th1.a}. This completes the proof of
Theorem \ref{th1}.

\renewcommand{\theequation}{\thesection.\arabic{equation}}
\setcounter{equation}{0}
 \section{The proof of Theorem \ref{thm5}}\label{sectcz}

\subsection{Outline of proof to  Theorem \ref{thm5}}\label{subsect4.1}\ The purpose of this section is to present the proof of
Theorem \ref{thm5} by following the same line  of that to Theorem
\ref{th1}. Toward this, we shall first construct the approximate
solutions to (\ref{1.3}) with data \eqref{1.6cz} as a perturbation
to the 2-D classical Navier-Stokes system with a parameter. Without
loss of generality, we may assume that the viscous coefficient
$\mu=1$ in \eqref{1.3}. The detailed strategy is as follows:

\no{\bf Step 1.}\ Construction of the approximate solutions.

As in \cite{c-g, GHZ1}, we denote $(v^h, \Pi_0)$ to be the global
smooth solution of the following 2-D Navier-Stokes system depending
on a parameter $y_3:$
\begin{equation*}
(NS2D_3)\quad
\begin{cases}
\pa_t v^h + v^h\cdot\na_h v^h - \Delta_h v^h =- \na_h \Pi_0, \\
\mathrm{div}_h v^h = 0, \\
v^h|_{t = 0} = v^h_0(\cdot, y_3).
\end{cases}
\end{equation*}
{\bf Notations}\ Here and in what follows, we always denote
$$x_h=(x_1,x_2),\quad\na_h=(\pa_{x_1},\pa_{x_2}),\quad \Delta_h=\pa_{x_1}^2+\pa_{x_2}^2,\quad
 \mbox{and}\quad [b]_\e(x)=b(x_h,\e x_3).$$

Then as in \cite{c-g} and\ccite{GHZ1}, we define the approximate solutions
$(v^{\ep}_{app}(t,x),\, \Pi^{\ep}_{app}(t,x))$ as
\begin{equation}\label{appsol}
\begin{split}
v^{\ep}_{app}(t,x) &\eqdefa (v^h,0)(t,x_h, \ep
x_3)\quad\mbox{and}\quad \Pi^{\ep}_{app}(t,x) \eqdefa \Pi_0(t,x_h,
\ep x_3),
\end{split}
\end{equation}
which satisfy
\begin{equation}\label{INSapp}
\begin{cases}
\big( \pa_t v^{\ep}_{app} + v^{\ep}_{app} \cdot \na v^{\ep}_{app} -
\Delta v^{\ep}_{app} + \na \Pi^{\ep}_{app} \big)(t,x_h,x_3) =
F^{\ep}(t,x_h,x_3),\\
\dive v^{\ep}_{app}=0,\\
v^{\ep}_{app}(t, x_h, x_3)|_{t=0}=u_0^{\ep}(x_h, x_3)\eqdefa
(v_0^h,0)(x_h,\ep x_3)
\end{cases}
\end{equation}
with
\begin{equation}\label{F-term-1}
\begin{split}
F^{\ep}(t,x_h,x_3) =\ep F_1(t,x_h,\ep x_3) + F_2^{\ep}(t,x_h,\ep
x_3),
\end{split}
\end{equation}
where
\begin{equation*}\label{F-term-1-a}
\begin{split}
&F_1(t,x_h, y_3) \eqdefa (0,\pa_3 \Pi_0)(t,x_h, y_3)\quad\mbox{
and}\quad F_2^{\ep}(t,x_h, y_3)  \eqdefa \ep^2(\pa_3^2 v^h,0)(t,x_h,
y_3).
\end{split}
\end{equation*}

\no{\bf Step 2.}\ The estimate of the error between the true
solution and the approximate ones.

Let \beq\label{1.7cz} R^{\ep} \eqdefa u^{\ep} - v^{\ep}_{app}\quad
\mbox{ and }\quad Q^{\ep} \eqdefa \Pi^{\ep} - \Pi^{\ep}_{app}.\eeq
Then it follows from (\ref{1.3}) and \eqref{INSapp} that  $(a^\ep,
R^\ep, Q^\ep)$ solves
\begin{equation}\label{remainder-system-1}
\begin{cases}
 \pa_t a^{\ep} + (R^{\ep} + v^{\ep}_{app} )\cdot\na a^{\ep} = 0,\qquad (t,x)\in\R^+\times\R^3, \\
\pa_t R^{\ep} + R^{\ep} \cdot \na R^{\ep} +R^{\ep} \cdot \na
v^{\ep}_{app}
+ v^{\ep}_{app} \cdot \na R^{\ep}- \Delta R^{\ep} + \na Q^{\ep}  \\
\qquad= a^{\ep}(\Delta R^{\ep} + \Delta v^{\ep}_{app} - \na Q^{\ep}
-
\na \Pi^{\ep}_{app} ) - F^{\ep},\\
\dive R^{\ep} =\dive v^{\ep}_{app}= 0,\\
 a^{\ep}(t,x_h,x_3)|_{t = 0} = \ep^{\delta_0} a_0(x_h,\ep x_3),\quad
 R^{\ep}|_{t = 0} = 0.
\end{cases}
\end{equation}
To solve \eqref{remainder-system-1} globally  in the framework of
the anisotropic Besov space $\fB^{0}_p,$ in general, one should
require $L^1(\mathbb{R}^{+}; \fB^{0}_p(\R^3))$ estimate for the
source term $ F^{\ep}$ given by \eqref{F-term-1}. Nevertheless,
according to Lemma \ref{le2.0cz} below (see also (4.7) of
\cite{GHZ1}), we do not have the $L^1(\mathbb{R}^{+};
\fB_{p}^{0}(\R^3))$ estimate for the term $F_2^\ep.$ To deal with
this term, as in \cite{GHZ1}, we denote \beno V^h(t,x_h,x_3) \eqdefa
(\pa_3 v^h,0)(t,x_h,\ep x_3),\eeno then $ F^{\ep}_2=\ep \pa_3 V^h$.
We shall  construct $(R_1^\ep, \, \Pi_{v}^\ep)$ via
\begin{equation}\label{remainder-system-2}
\begin{cases}
\pa_t R_1^\ep  - \Delta R_1^\ep +\nabla \Pi_{v}^\ep =-F^{\ep}_2, \\
\dive R_1^\ep = 0,\\
R_1^\ep|_{t = 0} = 0,
\end{cases}
\end{equation}
then as $\dive V^h=0,$ we have $\dive F_2^\ep=0$ and
$\na\Pi_v^\ep=0,$  so that let \beq \label{1.9cz} w^{\ep} \eqdefa
R^{\ep} - R_1^\ep, \eeq
 to
solve \eqref{remainder-system-1} for $(a^\ep, R^\ep, Q^\ep)$ is
reduced to solve $(a^{\ep}, \, w^\ep,\, Q^{\ep})$ through
\begin{equation}\label{1.13cz}
\begin{cases}
\pa_t a^{\ep} + (w^\ep +R_1^\ep+  v^{\ep}_{app} )\cdot\na a^{\ep}=0,\qquad (t,x)\in\R^+\times\R^3, \\
\pa_t w^\ep + w^\ep\cdot \na  w^\ep + w^\ep \cdot \na
(v^{\ep}_{app}+R_1^\ep) +
(v^{\ep}_{app}+R_1^\ep)\cdot \na w^\ep -\Delta w^\ep\\
\qquad + (1+a^\ep)\na Q^{\ep} =G^{\ep},\\
\dive w^\ep = 0,\\
 a^{\ep}(t,x_h,x_3)|_{t = 0} = \ep^{\delta_0}
a_0(x_h,\ep x_3),\quad w^\ep|_{t=0}=0,
\end{cases}
\end{equation}
with  $F_1^\ep$ given by \eqref{F-term-1} and
\begin{equation}
\label{1.10b}
\begin{split} G^{\ep} \eqdefa&
  a^{\ep}(\Delta w^{\ep} +\Delta R_1^{\ep}+\Delta v^{\ep}_{app}-\nabla \Pi_{app}^{\ep}\\
&\qquad -R_1^{\ep}\cdot \na (R_1^{\ep}+v^{\ep}_{app})
 -v^{\ep}_{app}\cdot \na R_1^{\ep} - \e[F_1]_\e.
 \end{split}
\end{equation}
We shall follow the same line of the proof of Theorem \ref{th1} to
construct the global solution of\refeq{1.13cz}. Namely, we shall
first estimate $a^\e$ in the isentropic Besov spaces
$$
\wt{L}^\infty(\R^+;\cB^{\f3p}_p(\R^3))\cap\wt{L}^\infty(\R^+;\cB^{-1+\f3q}_q(\R^3))
$$
for $q$ in~$]\f32,2[$ and~$p$ in~$]3,4[,$ and then we estimate $w^\e$ in
the anisotropic Besov spaces
$$
\wt{L}^\infty(\R^+;\fB^{0}_p(\R^3))\cap L^1(\R^+;\fB^{2}_p(\R^3)).
$$
With these estimates, we repeat the argument at the beginning of
Section 3 to construct the unique global solution of \eqref{1.3}
with data\refeq{1.6cz}.

\subsection{Technical Lemmas}

For simplicity, we shall neglect the subscript $\e$ in the rest of
this section. Let us first recall Lemma 3.2 and inequality~(4.7)
of\ccite{GHZ1}.

\begin{lem}
\label{le2.0cz}
{\sl Let $(v^h,\Pi_0)$ be a smooth enough solution of ($NS2D_3$) and
$R_1$ be determined by \eqref{remainder-system-2}. Then under the
assumptions of Theorem \ref{thm5}, for any $\alpha $ in~$\mathbb{N}^3$, one has
\begin{equation*}
\begin{split}
&\|\pa^{\alpha}v^h\|_{\widetilde{L}^{\infty}(\mathbb{R}^+;\fB_{2}^{0})}
+ \|\pa^{\alpha}v^h\|_{L^1(\mathbb{R}^+;\fB_{2}^{2,\frac{1}{2}})}
+\|\pa^{\al} \Pi_0\|_{L^1(\mathbb{R}^+;\fB_{2}^{0})} \leq C_{v_0},
\\
&\|\pa^{\al}\pa_3
v^h\|_{\widetilde{L}^{\infty}(\mathbb{R}^+;\fB_{2}^{-1,\frac{1}{2}})}
 + \|\pa^{\al}\pa_3 v^h\|_{L^1(\mathbb{R}^+;\fB_{2}^{1,\frac{1}{2}})}\leq
 C_{v_0},
\end{split}
\end{equation*}
and \beno
\begin{split}
\|R_1\|_{\widetilde{L}^{\infty}_t(\R^+;\fB^{0}_{2})}+\sum_{|\al|\leq
1}\|\p^\al\na R_1\|_{L^1(\R^+;\fB^{1,\frac{1}{2}}_{2})}\leq
C_{v_0}\ep.
\end{split}
\eeno }
\end{lem}

\begin{lem}\label{lem2.4cz}
{\sl Let  $q$ be in~$]1,2[,$~$p$ in~$ ]3,4[,$ and  $G$ be given by
\eqref{1.10b}. Then under the assumptions of Theorem \ref{thm5},  one
has \beq\label{2.7cz}
\begin{split}\|G\|_{L^1_t(\fB^{0}_p)}\lesssim &
C\|a\|_{L^\infty_t(\cB^{\f3p}_p)}\|\D w\|_{L^1_t(\fB^{0}_p)}+
C_{v_0}\bigl(\e+\|a\|_{L^\infty_t(\cB^{\f3p}_p)}+\e\|a\|_{L^\infty_t(\cB^{-1+\f3q}_q)}\bigr).
\end{split}
\eeq }
\end{lem}

\begin{proof} Notice that $p<4,$ applying Lemma \ref{propanisoiso} and Lemma \ref{lem2.2} yields
\beno \begin{split} \|a(\D w+\D_h R_1)\|_{L^1_t(\fB^{0}_p)}\lesssim&
\|a\|_{L^\infty_t(\fB^{\f2p,\f1p}_p)}\bigl(\|\D
w\|_{L^1_t(\fB^{0}_p)}+\|\D_h
R_1\|_{L^1_t(\fB^{0}_p)}\bigr)\\
\lesssim & \|a\|_{L^\infty_t(\cB^{\f3p}_p)}\bigl(\|\D
w\|_{L^1_t(\fB^{0}_p)}+\| R_1\|_{L^1_t(\fB^{1+\f2p,\f1p}_p)}\bigr).
\end{split}
\eeno Similarly as $q$ is in~$]1,2[$, $1-\f2q>0$ so that one has \beno
\begin{split} \|a\p_3^2
R_1\|_{L^1_t(\fB^{0}_p)}\lesssim&
\|a\|_{L^\infty_t(\fB^{-1+\f2q,\f1q}_q)}\|\p_3^2
R_1\|_{L^1_t(\fB^{\f2p,\f1p}_p)}\\
\lesssim & \|a\|_{L^\infty_t(\cB^{-1+\f3q}_q)}\|\p_3^2
R_1\|_{L^1_t(\fB^{\f2p,\f1p}_p)}.
\end{split}
\eeno
 It follows the same line that
 \beno
 \|a\D v_{app}\|_{L^1_t(\fB^{0}_p)}\lesssim \|a\|_{L^\infty_t(\cB^{\f3p}_p)}\|
v^h\|_{L^1_t(\fB^{1+\f2p,\f1p}_p)}+\e^2\|a\|_{L^\infty_t(\cB^{-1+\f3q}_q)}\|\p_3^2
v^h\|_{L^1_t(\fB^{\f2p,\f1p}_p)}, \eeno and \beno
\|a\na\Pi_{app}\|_{L^1_t(\fB^{0}_p)}\lesssim
\|a\|_{L^\infty_t(\cB^{\f3p}_p)}\|\na\Pi_0\|_{L^1_t(\fB^{0}_p)}.
\eeno Whereas applying Lemma \ref{lem2.2} twice leads to
$$
\|aR_1\cdot\na(R_1+v_{app})\|_{L^1_t(\fB^{0}_p)}\lesssim
\|a\|_{L^\infty_t(\cB^{\f3p}_p)}\|R_1\|_{L^\infty_t(\fB^{0}_p)}\bigl(\|\na
R_1\|_{L^1_t(\fB^{\f2p,\f1p}_p)}+\|\na
v^h\|_{L^1_t(\fB^{\f2p,\f1p}_p)}\bigr)
$$
 and
 $$ \|av_{app}\cdot\na R_1\|_{L^1_t(\fB^{0}_p)}\lesssim
\|a\|_{L^\infty_t(\cB^{\f3p}_p)}\|v^h\|_{L^\infty_t(\fB^{0}_p)}\|
R_1\|_{L^1_t(\fB^{1+\f2p,\f1p}_p)}.
$$
As a consequence, we
obtain
\beno
\begin{split}
\|G\|_{L^1_t(\fB^{0}_p)}\lesssim & \|a\|_{L^\infty_t(\cB^{\f3p}_p)}
\Bigl(\|\D w\|_{L^1_t(\fB^{0}_p)}+\|
R_1\|_{L^1_t(\fB^{1+\f2p,\f1p}_p)}+\|
v^h\|_{L^1_t(\fB^{1+\f2p,\f1p}_p)}+\|\na\Pi_0\|_{L^1_t(\fB^{0}_p)}\\
&\quad
{}+\bigl(\|R_1\|_{L^\infty_t(\fB^{0}_p)}+\|v^h\|_{L^\infty_t(\fB^{0}_p)}\bigr)
\bigl(\|\na R_1\|_{L^1_t(\fB^{\f2p,\f1p}_p)}+\|\na
v^h\|_{L^1_t(\fB^{\f2p,\f1p}_p)}\bigr)\Bigr)\\
&{}+\|a\|_{L^\infty_t(\cB^{-1+\f3q}_q)}\bigl(\|\p_3^2
R_1\|_{L^1_t(\fB^{\f2p,\f1p}_p)}+\e^2\|\p_3^2
v^h\|_{L^1_t(\fB^{\f2p,\f1p}_p)}\bigr)+\e\|\p_3\Pi_0\|_{L^1_t(\fB^{0}_p)},
\end{split}
\eeno from which, Lemma \ref{lem2.1} and   Lemma \ref{le2.0cz},
 we conclude the proof of \eqref{2.7cz}.
\end{proof}

\subsection{The proof of Theorem
\ref{thm5}} It follows from the argument in Subsection
\ref{subsect4.1} that we only need to solve \eqref{1.13cz} globally
for $\ep$ sufficiently small in order to prove  Theorem \ref{thm5}.
Given data \eqref{1.6cz}, it is well-known that \eqref{1.3} has a
unique local solution $(a,u)$ on $(0, T^\ast)$ for some $T^\ast>0.$
Without loss of generality, we may assume that $T^\ast$ is the
lifespan of $(a,u).$ Of course, the solution $(a^\ep,w^\ep,Q^\ep)$
of \eqref{1.13cz} entails this lifespan $T^\ast.$ Similar to the
proof of Theorem \ref{th1} in Section 3, we denote \beq\label{w.1q}
T^\clubsuit\eqdefa \sup\bigl\{ T<T^\ast\ /\ \eta_T\eqdefa
\|a\|_{L^\infty_T(\cB^{\f3p}_p)}+\|w\|_{L^\infty_T(\fB^0_p)}+\|w\|_{L^1_T(\fB^2_p)}\leq
\d\ \bigr\}, \eeq for some sufficiently small positive constant
$\d,$ which will be chosen later on.

We also define
 \beq \label{2.1cz}
 \begin{split}
f_{\la}(t)\eqdefa& f(t)\exp\Bigl(-\la\int_0^t
V_h(t')\,dt'\Bigr\}\quad\mbox{with}\quad V_h(t)\eqdefa
\|v^h(t)\|_{\cB^{\f2p,\f1p}_p}^2+\|v^h(t)\|_{\cB^{1+\f2p,\f1p}_p},\\
&\imath(a,w)\eqdefa w\cdot \na  w+ w\cdot \na
(v_{app}+R_1)+(v_{app}+R_1)\cdot \na w-G. \end{split}\eeq Then
thanks to \eqref{1.13cz}, $w_\la$ solves \beno\begin{split} \pa_t
w_\la+& \la V_h(t)w_\la
 -\Delta w_\la + (1+a)\na Q_\la+\imath(a,w)_\la
=0. \end{split} \eeno Applying $\D_j^h\D_k^v$ to the above equation,
then taking the $L^2$ inner product of the resulting equation with
$|\D_j^h\D_k^v w_{\la}|^{p-2}\D_j^h\D_k^v w_{\la}$ for $p$ in~$]3,4[$
and integrating the resulting equation over $[0,t],$ we obtain
 \beno
\begin{split}
\|\D_j^h\D_k^vw_\la\|_{L^\infty_t(L^p)}
&+\la\int_0^tV_h(t')\|\D_j^h\D_k^vw_\la\|_{L^p}\,dt'\\
&+c(2^{2j}+2^{2k})\|\D_j^h\D_k^vw_\la\|_{L^1_t(L^p)}\leq \bigl \|
\bigr((1+a)\na Q -\imath(a,w)\bigr)_\la\bigr\|_{L^1_t(L^p)},
\end{split} \eeno for some $c>0.$ After summation, this gives
\beq\label{2.2cz}
\begin{split}
\|w_\la\|_{\wt{L}^\infty_t(\fB^{0}_p)}+\la\int_0^tV_h(t')\|w_\la(t')\|_{\fB^{0}_p}\,dt'
&+c\|w_\la\|_{L^1_t(\fB^{2}_p)}
\\
&\lesssim \bigl\|\bigr((1+a)\na Q
-\imath(a,w)\bigr)_\la\bigr\|_{L^1_t(\fB^{0}_p)}.
\end{split} \eeq
Lemma \ref{lemmapressureinhomo} and \eqref{1.13cz} implies \beno
\bigl\|\bigl((1+a)\na Q
-\imath(a,w)\bigr)_\la\bigr\|_{\fB^{0}_p}\lesssim
\|(\imath(a,w))_\la\|_{\fB^{0}_p}. \eeno
 And as $p<4,$ applying Lemma
\ref{lem2.2} leads to \beno \begin{split} &\|w\cdot\na
w_\la\|_{\fB^{0}_p}\lesssim \|w\|_{\fB^{0}_p}\|\na
w_\la\|_{\fB^{\f2p,\f1p}_p}\lesssim \|w\|_{\fB^{0}_p}\|
w_\la\|_{\fB^{2}_p},\\
&\|w_\la\cdot\na R_1\|_{\fB^{0}_p}\lesssim \|\na
R_1\|_{\fB^{\f2p,\f1p}_p}\|w_\la\|_{\fB^{0}_p},\\
 &\|R_1\cdot\na w_\la\|_{\fB^{0}_p}\lesssim
\|R_1\|_{\fB^{0}_p}\| w_\la\|_{\fB^{2}_p}. \end{split} \eeno While
notice that for $[b]_\e(x)=b(x_h,\e x_3),$ \beno w_\la\cdot\na
v_{app}=w^h_\la\cdot[\na_h v^h]_\ep+\e w^3_\la[\p_3v^h]_\e, \eeno we
get, by applying Lemma \ref{lem2.2} once again, that
\beno\begin{split} \|w_\la\cdot\na v_{app}\|_{\fB^{0}_p}\lesssim &
\|v^h\|_{\fB^{1+\f2p,\f1p}_p}\|w_\la^h\|_{\fB^{0}_p}+\e\|\p_3v^h\|_{\fB^{\f2p,\f1p}_p}
\|w_\la^3\|_{\fB^{0}_p}. \end{split} \eeno Along the same line, one
has \beno
\begin{split}
&\|v_{app}\cdot\na w_\la\|_{\fB^{0}_p)}\lesssim
\|v^h\|_{\fB^{\f2p,\f1p}_p}\|\na_h w_\la\|_{\fB^{0}_p}.
\end{split}
\eeno

Therefore, substituting the above estimates into \eqref{2.2cz}, we infer that
 for any~$t\leq T^\clubsuit$,
\beq
\label{2.4cz}
\begin{split} &\|w_\la\|_{\wt{L}^\infty_t(\fB^{0}_p)}+\la\int_0^tV_h(t')\|w_\la(t')\|_{\fB^{0}_p}\,dt'
+c\|w_\la\|_{L^1_t(\fB^{2}_p)}\\
&\lesssim \|G_\la\|_{L^1_t(\fB^{0}_p)}+\bigl(\|\na
R_1\|_{L^1_t(\fB^{\f2p,\f1p}_p)}+\e\|\p_3v^h\|_{L^1_t(\fB^{\f2p,\f1p}_p)}\bigr)\|w_\la\|_{L^\infty_t(\fB^{0}_p)}\\
&\qquad{}+
\bigl(\|w\|_{L^\infty_t(\fB^{0}_p)}+\|R_1\|_{L^\infty_t(\fB^{0}_p)}\bigr)
\|
w_\la\|_{L^1_t(\fB^{2}_p)}\\
&\qquad{}+\int_0^t\bigl(\|v^h\|_{\fB^{\f2p,\f1p}_p}\|
w_\la\|_{\fB^{\f2p,\f1p}_p}+\|v^h\|_{\fB^{1+\f2p,\f1p}_p}\|w_\la\|_{\fB^{0}_p}\bigr)\,dt'.
\end{split}
\eeq Whereas it follows from the simple interpolation in the
anisotropic Besov spaces that \beno
\begin{split}
\int_0^t\|v^h\|_{\fB^{\f2p,\f1p}_p}\|
w_\la\|_{\fB^{\f2p,\f1p}_p}\,dt'\lesssim
\Bigl(\int_0^t\|v^h\|_{\fB^{\f2p,\f1p}_p}^2\|
w_\la\|_{\fB^{0}_p}\,dt'\Bigr)^{\f12}\|
w_\la\|_{L^1_t(\fB^{2}_p)}^{\f12}, \end{split} \eeno from which and
\eqref{2.4cz}, we infer for $t\leq  T^\clubsuit$
 \beno
\begin{split}
 &\|w_\la\|_{\wt{L}^\infty_t(\fB^{0}_p)}+\la\int_0^tV_h(t')\|w_\la(t')\|_{\fB^{0}_p}\,dt'
+\f{c}2\|w_\la\|_{L^1_t(\fB^{2}_p)}\\
&\qquad\lesssim \|G_\la\|_{L^1_t(\fB^{0}_p)}+\bigl(\|\na
R_1\|_{L^1_t(\fB^{\f2p,\f1p}_p)}+\e\|\p_3v^h\|_{L^1_t(\fB^{\f2p,\f1p}_p)}\bigr)\|w_\la\|_{L^\infty_t(\fB^{0}_p)}\\
&\qquad\qquad\quad+
\bigl(\|w\|_{L^\infty_t(\fB^{0}_p)}+\|R_1\|_{L^\infty_t(\fB^{0}_p)}\bigr) \|
w_\la\|_{L^1_t(\fB^{2}_p)}+\int_0^tV_h(t')\|w_\la(t')\|_{\fB^{0}_p}\,dt',
\end{split}
\eeno for $V_h(t)$ defined by \eqref{2.1cz}. Taking $\la\geq C$ in
the above inequality and applying Lemma \ref{le2.0cz} and Lemma
\ref{lem2.4cz}, we obtain  for $t\leq  T^\clubsuit$
\beq\label{2.5cz}
\begin{split} \|w_\la\|_{\wt{L}^\infty_t(\fB^{0}_p)}
&+\f{c}2\|w_\la\|_{L^1_t(\fB^{2}_p)} \leq
C_{v_0}\bigl(\e+\|a\|_{L^\infty_t(\cB^{\f3p}_p)}+\e\|a\|_{L^\infty_t(\cB^{-1+\f3q}_q)}\\
&+\e\|w_\la\|_{\wt{L}^\infty_t(\fB^{0}_p)}\bigr)
+C\bigl(\|a\|_{L^\infty_t(\cB^{\f3p}_p)}+\|w\|_{\wt{L}^\infty_t(\fB^{0}_p)}+C_{v_0}\e\bigr)\|
w_\la\|_{L^1_t(\fB^{2}_p)}.
\end{split}
\eeq
Then taking $\d\leq \f{c}{8C}$ in \eqref{w.1q} and
$\e\leq\min\bigl\{\f{c}{8CC_{v_0}}, \f{1}{2C_{v_0}}\bigr\},$ we deduce
from \eqref{2.5cz} that \beq\label{2.10cz}
\begin{split}
\|w_\la&\|_{\wt{L}^\infty_t(\fB^{0}_p)}+\f{c}4\|w_\la\|_{L^1_t(\fB^{2}_p)}
\leq
2C_{v_0}\bigl(\e+\|a\|_{L^\infty_t(\cB^{\f3p}_p)}+\e\|a\|_{L^\infty_t(\cB^{-1+\f3q}_q)}\bigr)
\end{split}
\eeq for $t\leq T^\clubsuit.$

On the other hand, applying Lemma \ref{lem3.2cz} to the free
transport equation in \eqref{1.13cz} that for any $s$ in~$]0,1[$,
\beq\label{2.13cz}
\begin{split}
\|a\|_{L^\infty_t(\cB^s_p)}\leq &
\|a_{0,\e}\|_{\cB^s_p}\exp\Bigl(\|\na
R_1\|_{L^1_t(\cB^{\f2p,\f1p}_p)}+\|\na
w\|_{L^1_t(\cB^{\f2p,\f1p}_p)}+\|\na
v^h\|_{L^1_t(\cB^{\f2p,\f1p}_p)}\Bigr)\\
\leq &C_{v_0}\e^{\s-\f1p}\|a_0\|_{\cB^s_p} \quad\mbox{for}\quad
t\leq T^\clubsuit.
\end{split}
\eeq As $q$ is in~$]\f32,2[,$ we can apply his result
with~$-1+\f3q.$  Together with \eqref{2.10cz}  this ensures that,
for any~$t\leq T^\clubsuit$, \beq \label{2.11cz}
\begin{split}
\|w_\la&\|_{\wt{L}^\infty_t(\fB^{0}_p)}+\f{c}4\|w_\la\|_{L^1_t(\fB^{2}_p)}
\leq
2C_{v_0}\bigl(\e+\e^{\s-\f1p}\|a_0\|_{\cB^{\f3p}_p}+\e^{1+\s-\f1q}\|a_0\|_{\cB^{-1+\f3q}_q}\bigr)\end{split}
\eeq

 By virtue of \eqref{2.1cz} and
\eqref{2.11cz}, we obtain \beno
\begin{split}
\|w\|_{\wt{L}^\infty_t(\fB^{0}_p)}+\f{c}4\|w\|_{L^1_t(\fB^{2}_p)}
\leq&
2C_{v_0}(\e^{\s-\f1p}\|a_0\|_{\cB^{\f3p}_p}+\e^{1+\s-\f1q}\|a_0\|_{\cB^{-1+\f3q}_q})\\
&\quad\times\exp\Bigl(C\bigl(
\|v^h\|_{L^1_t(\fB^{1+\f2p,\f1p}_p)}+\|v^h\|_{L^2_t(\fB^{\f2p,\f1p}_p)}^2\bigr)\Bigr)\\
\leq&
\bar{C}_{v_0}\bigl(\e^{\s-\f1p}\|a_0\|_{\cB^{\f3p}_p}+\e^{1+\s-\f1q}\|a_0\|_{\cB^{-1+\f3q}_q}\bigr)
\quad\mbox{for} \quad t\leq T^\clubsuit.
\end{split}
\eeno Now as $\s>\f14,$ we can take $p_\s<4$ so that
$\s-\f1{p_\s}>0.$ Then for $\e$ small enough, we conclude that
$T^\ast=T^\clubsuit,$ and there holds \beq\label{2.12cz}
\|w\|_{\wt{L}^\infty_t(\fB^{0}_{p_\s})}+\f{c}4\|w\|_{L^1_t(\fB^{2}_{p_\s})}
\leq \bar{C}_{a_0,v_0}\e^{\s-\f1{p_\s}}\quad\mbox{for} \quad t\leq
T^\ast. \eeq With \eqref{2.13cz} and \eqref{2.12cz}, it is standard
to prove that $T^\ast=\infty$ and the global solution $(a,u)$ of
\eqref{1.3} such that
$$
a\in \cC([0,\infty);
\cB^{\f3{p_\s}}_{p_\s}(\R^3)) \cap \wt{L}^\infty(\R^+;
\cB^{\f3{p_\s}}_{p_\s}(\R^3))
$$ and
$$u \in \cC([0,\infty);
\cB^{-1+\f3{p_\s}}_{p_\s}(\R^3))\cap
\wt{L}^\infty(\R^+;\cB^{-1+\f3{p_\s}}_{p_\s}(\R^3)) \cap
L^1(\R^+;\cB^{1+\f3{p_\s}}_{p_\s}(\R^3)).
$$  The uniqueness part is
guaranteed by Theorem 1 of \cite{DM}. We thus complete the proof of
Theorem \ref{thm5}.

\bigskip

\noindent {\bf Acknowledgments.} Part of this work was done when we
were visiting Morningside Center of the Chinese Academy of Sciences.
We appreciate the hospitality and the financial support from MCM. P.
Zhang is partially supported by NSF of China under Grant 10421101
and 10931007, the one hundred talents' plan from Chinese Academy of
Sciences under Grant GJHZ200829 and innovation grant from National
Center for Mathematics and Interdisciplinary Sciences.

\end{document}